\RequirePackage{fix-cm}
\documentclass{svjour3}                     
\smartqed  
\usepackage{graphicx}
\usepackage{amsmath,amsfonts,amssymb}
\usepackage{color}
\usepackage{listings}
\newcommand{\NZ}{\ensuremath{\setminus \{0\}}}

\usepackage[usenames,dvipsnames]{xcolor}
\definecolor{DarkBlue}{rgb}{0,0,0.4}
\definecolor{DarkRed}{rgb}{0.3,0,0}
\definecolor{DarkGreen}{rgb}{0,0.3,0}
\usepackage{xspace}

\usepackage{ifthen}
\let\Ooooooldcite\cite
\renewcommand{\cite}[2][\empty]{\ifthenelse{\equal{#2}{?}}{{\color{red}\Ooooooldcite[find me!]{?}}}{\ifthenelse{\equal{#1}{\empty}}
{\Ooooooldcite{#2}}{\Ooooooldcite[#1]{#2}}}}

\newcommand{\IN}{\ensuremath{\mathbb{N}}}
\newcommand{\IR}{\ensuremath{\mathbb{R}}}

\DeclareMathOperator*{\argmin}{arg\,min}
\DeclareMathOperator*{\esssup}{ess\,sup}
\DeclareMathOperator*{\essinf}{ess\,inf}

\newcommand{\BIGOP}[1]{\mathop{\mathchoice{\raise-0.22em\hbox{\Huge $#1$}}{\raise-0.05em\hbox{\Large $#1$}}{\hbox{\large$#1$}}{#1}}}

\newcommand{\tensor}{\otimes}

\usepackage[mathscr]{euscript}

\newcommand{\T}{\mathsf{T}}

\newcommand{\Tcal}{\ensuremath{\mathcal{T}}}

\newcommand{\Abf}{\ensuremath{\mathbf{A}}}
\newcommand{\Bbf}{\ensuremath{\mathbf{B}}}
\newcommand{\Cbf}{\ensuremath{\mathbf{C}}}
\newcommand{\Dbf}{\ensuremath{\mathbf{D}}}

\newcommand{\Ibf}{\ensuremath{\mathbf{I}}}
\newcommand{\Mbf}{\ensuremath{\mathbf{M}}}
\newcommand{\Nbf}{\ensuremath{\mathbf{N}}}

\newcommand{\Sbf}{\ensuremath{\mathbf{S}}}
\newcommand{\Tbf}{\ensuremath{\mathbf{T}}}

\newcommand{\Vbf}{\ensuremath{\mathbf{V}}}
\newcommand{\Xbf}{\ensuremath{\mathbf{X}}}

\newcommand{\bbf}{\ensuremath{\mathbf{b}}}
\newcommand{\dbf}{\ensuremath{\mathbf{d}}}
\newcommand{\ebf}{\ensuremath{\mathbf{e}}}

\newcommand{\sbf}{\ensuremath{\mathbf{s}}}
\newcommand{\ubf}{\ensuremath{\mathbf{u}}}
\newcommand{\vbf}{\ensuremath{\mathbf{v}}}
\newcommand{\wbf}{\ensuremath{\mathbf{w}}}

\newcommand{\zbf}{\ensuremath{\mathbf{z}}}
\newcommand{\zerobf}{\ensuremath{\mathbf{0}}}

\newcommand{\TEXT}[1]{\ensuremath{\quad\text{#1}\quad}}

\newcommand{\wrt}{with respect to }
\newcommand{\ie}{i.e.,\ }
\newcommand{\eg}{e.g.\ }
\newcommand{\cf}{cf.\ }

\newcommand{\Sec}{Section\xspace}
\newcommand{\Secs}{Sections\xspace}
\newcommand{\Fig}{Fig.\xspace}

\newcommand{\Thm}{Theorem\xspace}
\newcommand{\Lem}{Lemma\xspace}

\newcommand{\Pro}{Proposition\xspace}
\newcommand{\Pros}{Propositions\xspace}

\newcommand{\Cha}{Chapter\xspace}

\providecommand{\VERT}{\ensuremath{| \! | \! |}}

\newcommand{\Set}[1]{\left\{{#1}\right\}}
\newcommand{\set}[1]{\{{#1}\}}

\newcommand{\norm}[2]{\|{#1}\|_{{#2}}}
\newcommand{\tnorm}[2]{\VERT{#1}\VERT_{{#2}}}

\newcommand{\scalar}[2]{({#1})_{ {#2} }}

\newcommand{\DIV}{\mathop{\mathrm{div}}}
\newcommand{\GRAD}{\mathop{\mathrm{grad}}}

\newcommand{\Dt}{\partial_t}

\newcommand{\SubAndSup}[3]{%
	\ifthenelse{\equal{#3}{}}{%
		\ifthenelse{\equal{#2}{}}{%
			{#1}%
		}{%
			{#1}_{#2}%
		}
	}{%
		\ifthenelse{\equal{#2}{}}{%
			{#1}\ifthenelse{\equal{#3}{'}}{'}{^{#3}}%
		}{%
			{#1}_{#2}\ifthenelse{\equal{#3}{'}}{'}{^{#3}}%
		}%
	}%
}

\journalname{Numerical Algorithms}

\title{Space-time discretization of the heat equation}
\subtitle{A concise Matlab implementation}

\author{Roman Andreev}

\institute{%
	R.~Andreev
	\at
	RICAM,
	Altenberger-Str.~69, 4040 Linz, Austria
	\\
	\email{roman.andreev@oeaw.ac.at}
}

\let\olddate\date
\usepackage{ifthen}
\renewcommand{\date}[1]{\ifthenelse{\equal{#1}{}}{\def\makeheadbox{}}{\olddate{#1}}}

\date{%
}

\begin{document}

\maketitle

\begin{abstract}
	A concise Matlab implementation
	of a stable parallelizable 
	space-time Petrov-Galerkin
	discretization
	for parabolic evolution equations
	is given.
	Emphasis is on the reusability
	of spatial finite element codes.
	
\end{abstract}


\section{Introduction}

The spectrum of numerical methods for 
parabolic evolution equations is extremely broad,
which attests to the ubiquity and the relevance of 
such equations.
With the aim of developing
reliable massively parallel algorithms,
\eg for optimization problems
constrained by parabolic evolution equations,
several attempts have been made
to go beyond time-stepping methods,
see \cite[\Sec 5.1]{AndreevDiss} for
a modest attempt of an overview.
In this paper
we give a concise Matlab implementation, 
partly motivated by \cite{AlbertyCarstensenFunken1999},
of 
a specific space-time Petrov-Galerkin discretization
for parabolic evolution equations
from \cite{Andreev2013,AndreevDiss},
hoping
to provide a basis for possible further developments.
Spanning just a few lines of Matlab code,
it is parallelizable and stable in the Petrov-Galerkin sense,
which already distinguishes it from conventional methods 
for parabolic evolution equations.
Stability implies quasi-optimality of
the discrete solution in the natural solution spaces,
and is an important property in the resolution
of nonlinear problems.
Moreover,
the implementation is modular \wrt the spatial discretization,
admits time-dependent inputs and nonuniform temporal grids.
Since the algorithm is based on 
an iterative solution of a \emph{single} linear system,
another significant advantage to conventional methods is 
the ability to terminate the iteration
at a specified global accuracy.

The model parabolic evolution equation
under consideration is presented in \Sec \ref{s:model}
and is restated in a variational formulation.
The space-time Petrov-Galerkin discrete trial and test spaces 
that will be used to discretize the variational formulation
are introduced in \Sec \ref{s:sttp}.
In order to obtain stable algorithms
we develop in \Sec \ref{s:mrpg}
a generalization of the usual variational framework
for linear operator equations,
called the minimal residual Petrov-Galerkin discretization.
Choosing bases on the discrete trial and test spaces,
it leads to a linear system of generalized Gau{\ss} normal equations
along with a natural preconditioner.
In that framework,
certain norm-inducing operators play an important role. 
Specifically for the parabolic evolution problem,
these are defined in \Sec \ref{s:mn}.
In \Sec \ref{s:matrix}
we detail the Kronecker product structure
of the parabolic space-time operator and the norm-inducing operators
when assembled using space-time tensor product bases,
and comment on the data structures employed.
In the solution process,
the inverses of the matrix representations (of norm-inducing operators)
are required.
Their Kronecker product structure
is discussed in \Sec \ref{s:i:inv}.
The assembly procedure for the space-time source
is in \Sec \ref{s:i:b}.
A generalization of 
the LSQR algorithm of Paige and Saunders \cite{PaigeSaunders1982}
for the iterative resolution of the linear system
is given in \Sec \ref{s:i:glsqr}.
With those preparations,
the Matlab implementation is presented in \Sec \ref{s:code}.
Two numerical experiments are given in \Sec \ref{s:x}.
We conclude and point out some limitations in \Sec \ref{s:end}.

\section{Model problem and its space-time variational formulation} 
\label{s:model}

\newcommand{\domdim}{n}

Let $D \subset \IR^\domdim$, $\domdim \in \set{ 1, 2, 3 }$,
be a bounded connected domain
with a polyhedral boundary $\Gamma = \partial D$.
If $\domdim = 1$ then $D$ is an open bounded interval;
if $\domdim = 2$ then $D$ is a polygon;
if $\domdim = 3$ then $D$ is a polyhedron.
We partition ${\Gamma}$ into two disjoint subsets
${\Gamma_0}$ and ${\Gamma_N}$ such that 
$\bar\Gamma = \bar\Gamma_0 \cup \bar\Gamma_N$.
On ${\Gamma_0}$ we will impose 
homogeneous boundary conditions of Dirichlet type.
For that reason, 
the Dirichlet boundary $\Gamma_0$
is assumed
to be of positive measure
(\wrt the surface measure
which will subsequently be denoted by $\sigma$),
\ie it contains at least one endpoint if $\domdim = 1$;
a curve of positive length if $\domdim = 2$;
or
a surface of positive surface measure if $\domdim = 3$.
Let $J = (0, {T})$, ${T} > 0$, denote the temporal interval.

The model 
for parabolic evolution equations
that we consider
is the heat equation:
\begin{align}
\label{e:pde:u}
	\Dt u(t, x) - \DIV (a(x) \GRAD u(t, x)) & = f(t, x),
	&
	(t, x) & \in J \times D,
	\\
\label{e:pde:D}
	u(t, x) & = 0,
	&
	(t, x) & \in J \times {\Gamma_0},
	\\
\label{e:pde:N}
	a(x) \frac{\partial u}{\partial n}(t, x) & = g(t, x),
	&
	(t, x) & \in J \times {\Gamma_N},
	\\
\label{e:pde:0}
	u(t, x) & = h(x),
	&
	(t, x) & \in \{ 0 \} \times D
	.
\end{align}

Here, $a$, $f$, $g$ and $h$ are given scalar valued functions,
while $u$ is the unknown.
Further,
$\DIV$ and $\GRAD$
denote the divergence and the gradient operator
with respect to the spatial variable $x \in D$.
The derivative in the direction of the outward normal
at 
the Neumann part ${\Gamma_N}$ of the boundary
is denoted by $\frac{\partial u}{\partial n}$.
The precise meaning 
of the heat equation
will be fixed
by means of a well-posed space-time variational formulation
in the following.
The time-independent operator
``$\DIV (a(x) \GRAD u(t, x))$''
could be replaced
by the time-dependent one
``$\DIV( \Abf(t, x) \GRAD u(t, x) )$'',
where $\Abf \in L^\infty(J \times D; \IR_{\mathrm{sym}}^{\domdim \times \domdim})$,
without affecting most considerations below
(the technical reason
why this is possible is 
given in \cite[\Lem 4.4.1]{Fattorini2005}).
However,
if $\Abf$ is not a finite sum of separable functions,
the implementation becomes significantly less transparent,
and we therefore discard this case from the onset on;
on the other hand,
$\Abf$ being a finite sum of separable functions
entails modifications of secondary relevance to this exposition.

To motivate the space-time variational formulation
of the heat equation,
we formally test
the equation with a function $v_1$ on $J \times D$
and integrate in space and time;
the initial condition is tested with $v_2$ on $D$
and integrated in space.
Integration by parts is performed in space only,
and the two resulting conditions
(one ``$\forall v_1$'', the other ``$\forall v_2$'')
are added together.
The solution $u$ will be sought in the space $X$,
and the test functions are combined to 
$v = (v_1, v_2) \in Y := Y_1 \times Y_2$.
The spaces $X$ and $Y$ will be specified presently.
We write $\scalar{\cdot, \cdot}{D}$
for the $L^2(D)$ and the $[L^2(D)]^d$ scalar product,
while $\scalar{\cdot, \cdot}{\Gamma_N}$
is the scalar product on $L^2(\Gamma_N)$
for the boundary measure $\sigma$ introduced above.
Generally, we omit the dependence of the integrands on
the temporal variable $t$.
In this way we obtain
the continuous space-time variational formulation
\begin{align}
\label{e:vf}
	\text{find}
	\quad
	u \in X 
	\quad \text{such that} \quad
	B(u, v) = b(v)
	\quad
	\forall v \in Y,
\end{align}
where
the system bilinear form $B$ 
encoding the heat equation is 
\begin{align}
\label{e:B}
	B(u, v) :=
	\int_J \scalar{\Dt u, v_1}{D} dt
	+
	\int_J \scalar{a \GRAD u, \GRAD v_1}{D} dt
	+
	\scalar{u(0, \cdot), v_2}{D}
\end{align}
and
the load functional $b$
supplying the source term, as well as boundary and initial data,
is 
\begin{align}
\label{e:b}
	b(v) :=
	\int_J \scalar{f, v_1}{D} dt
	+
	\int_J \scalar{g, v_1}{\Gamma_N} dt
	+
	\scalar{ h, v_2 }{D}
	.
\end{align}

Integrating by parts also in time,
as in \eg \cite{BabuskaJanik1989,BabuskaJanik1990},
leads to
an alternative space-time variational formulation,
which, however,
will not be discussed below.

Let us introduce the abbreviations
\begin{align}
	V := H_{\Gamma_0}^1(D),
	\quad
	H := L^2(D)
	\TEXT{and}
	V' = [H_{\Gamma_0}^1(D)]',
\end{align}
where
$H_{\Gamma_0}^1(D)$ 
denotes the Sobolev space
of functions in $H^1(D)$ with vanishing trace on 
the Dirichlet boundary $\Gamma_0$,
and $[H_{\Gamma_0}^1(D)]'$ its dual.
We identify $H$ with its dual $H'$
by means of the Riesz isomorphism on $H$.
Then the duality pairing on $V \times V'$ (or $V' \times V$)
is
the continuous extension
of the $H$-scalar product on $V \times V$.
In this way we obtain a so-called Gelfand triple
of separable Hilbert spaces
\begin{align}
	V \hookrightarrow H 
	\cong 
	H' \hookrightarrow V'
\end{align}
with continuous and dense embeddings.
In order for
$B : X \times Y \to \IR$
to be a continuous bilinear form
and
for
$b : Y \to \IR$
to be a continuous linear functional,
it is now natural to take
\begin{align}
	X := L^2(J; V) \cap H^1(J; V')
	\quad\text{and}\quad
	Y := L^2(J; V) \times H
	.
\end{align}
For the definition of
the Bochner spaces $L^2(J; V)$, $H^1(J; V')$, and the like,
we refer to \eg \cite{Evans1998} or \cite[\Cha 1]{LionsMagenes1972}.
Then \eqref{e:B}--\eqref{e:b}
are well-defined for all $(u,v) \in X \times Y$
whenever
\begin{align}
	f \in L^2(J; V'),
	\quad
	g \in L^2(J; [H^{1/2}(\Gamma_N)]')
	\TEXT{and}
	h \in L^2(D)
	,
\end{align}
and
\begin{align}
\label{e:a}
	a \in L^\infty(D)
	\TEXT{with}
	0 < \essinf a \leq \esssup a < \infty
	.
\end{align}
For the remainder of the article
we assume
$f \in L^2(J; L^2(D))$.
The spaces $X$ and $Y$
are themselves Banach spaces for the norms
$\norm{ \cdot }{X}$ and $\norm{ \cdot }{Y}$
that are given by
\begin{align}
	\norm{ u }{X}^2 
	& :=
	\norm{ u }{ L^2(J; V) }^2 + \norm{ \Dt u }{ L^2(J; V') }^2,
	& & 
	u \in X,
\\
	\norm{ v }{Y}^2 
	& := 
	\norm{ v_1 }{ L^2(J; V) }^2 + \norm{ v_2 }{H}^2,
	& &
	v = (v_1, v_2) \in Y
	.
\end{align}
We recall
from \eg \cite[\Sec 5.9.2]{Evans1998} or \cite[\Cha 1]{LionsMagenes1972}
that
any (representant of any) $u \in X$
admits a modification on a negligible subset of $J$
such that the resulting function
coincides 
with a unique continuous $H$-valued function 
defined on the closed interval $\bar{J}$;
moreover,
the $C^0$ norm of the latter
is controlled by the $X$ norm of the former.
In other words,
the following embedding is continuous
\begin{align}
	X = L^2(J; V) \cap H^1(J; V') 
	\hookrightarrow
	C^0(\bar{J}; H),
\end{align}
and in particular
\begin{align}
	\exists C > 0:
	\quad
	\norm{ u(0) }{H} \leq C \norm{ u }{X}
	\quad
	\forall u \in X.
\end{align}
In this way, 
the initial value $u(0)$ of any $u \in X$
is well-defined in $H$.

With the above assumptions,
the space-time variational problem \eqref{e:vf}
has a unique solution $u \in X$ 
and the solution depends continuously on the functional $b \in Y'$,
see
\cite[\Thm 5.1]{SchwabStevenson2009}
and the references therein.
Hence, the solution $u$ also depends continuously 
on the input data $f$, $g$ and $h$ that 
define the load functional $b$ in \eqref{e:b}.

\section{Space-time tensor product discrete trial and test spaces}
\label{s:sttp}

The continuous space-time variational formulation \eqref{e:vf}
will be discretized
using
finite-dimensional discrete trial and test spaces
$X_h \subset X$ and $Y_h \subset Y$
built up from 
finite-dimensional ``univariate''
temporal subspaces
$E \subset H^1(J)$, $F \subset L^2(J)$,
and
spatial subspaces
$V_h \subset V$.
These spaces assume 
the space-time tensor product form
\begin{align}
\label{e:XhYh}
	X_h := E \tensor V_h
	\TEXT{and}
	Y_h := (F \tensor V_h) \times V_h
	.
\end{align}

A key feature of 
this discretization and 
the implementation given below
is the modularity with respect to
the spatial subspaces $V_h$.

To specify 
the temporal subspaces $E$ and $F$
we need to introduce some terminology.
A temporal mesh $\Tcal$ is 
a finite set of points in $\bar{J} = [0, T]$ containing $0$ and $T$.
The connected components of $J \setminus \Tcal$
are called 
the elements of $\Tcal$.
Let
$\max \Delta \Tcal$ denote
the maximal ``time-step'',
\ie the maximal length of an element of $\Tcal$.
For a temporal mesh $\Tcal$
let $\Tcal^\star$
denote the temporal mesh
obtained from $\Tcal$
by a uniform refinement
(each element is split into two smaller elements of equal length).
Concerning 
$E$ on the trial side and $F$ on the test side,
we will restrict ourselves to
two types of pairs
that differ in the choice of $F$.
\begin{description}
\item[\bf Type 1 temporal subspaces.]
	Given a temporal mesh $\Tcal_E$,
	we define $E$ as
	the standard space of continuous piecewise affine functions,
	and $F$ as
	the space of piecewise constant functions
	on $\Tcal_F := \Tcal_E$.
\item[\bf Type 2 temporal subspaces.]
	Given a temporal mesh $\Tcal_E$,
	the space $E$ is defined as above.
	Let another temporal mesh, $\Tcal_F$ be 
	obtained from $\Tcal_E$
	by a succession of uniform refinements,
	\ie $\Tcal_F = [\Tcal \mapsto \Tcal^\star]^n(\Tcal_E)$
	for some positive $n \in \IN$.
	Then $F$ is defined as
	the space of piecewise constant functions on $\Tcal_F$.
\end{description}

In the first case
the discrete variational formulation
\begin{align}
\label{e:uh0}
	\text{find}
	\quad
	u_h \in X_h
	\quad \text{such that} \quad
	B(u_h, v_h) = b(v_h)
	\quad
	\forall v_h \in Y_h
\end{align}
is 
an example of
continuous Galerkin time-stepping schemes
\cite{Hulme1972-118,AkrivisMakridakisNochetto2011}.
In the second case
the dimension of $Y_h$ is larger than that of $X_h$,
and 
the above discrete variational formulation 
is meaningless.
A generalization 
based on residual minimization 
is therefore introduced in \Sec \ref{s:mrpg}.
Concerning 
the stability of
the resulting minimal residual Petrov-Galerkin method
there is a fundamental difference
between Type 1 and Type 2
temporal subspaces.
This is the subject of the following two propositions
that summarize
the relevant main results from \cite[\Sec 5.2.3]{AndreevDiss}.
Note carefully
that the present concept of stability,
namely the validity of 
the discrete inf-sup condition 
\begin{align}
\label{e:is0}
	\gamma_h :=
	\inf_{u_h \in X_h \NZ}
	\sup_{v_h \in Y_h \NZ}
	\frac{
		B(u_h, v_h)
	}{
		\norm{ u_h }{X}
		\norm{ v_h }{Y}
	}
	> 0
\end{align}
(its role is discussed in \Sec \ref{s:mrpg})
uniformly in the choice of the temporal discretization,
is
different 
from \eg A-stability for time-stepping methods.

The following measure of self-duality
for the spatial subspace $V_h \subset V$
will be needed:
\begin{align}
	\kappa_h :=
	\inf_{ \chi_h' \in V_h \NZ }
	\sup_{ \chi_h  \in V_h \NZ }
	\frac
	{ \scalar{ \chi_h', \chi_h }{D} }
	{ \norm{ \chi_h' }{V'} \norm{ \chi_h }{V} }
	.
\end{align}
Note that $\kappa_h$ is bounded (independently of $V_h$)
and 
necessarily positive for a finite-dimensional $V_h$.

\begin{proposition}
\label{p:T2}
	Let $\set{0} \neq V_h \subset V$ be a finite-dimensional subspace.
	Let $E \subset H^1(J)$ and $F \subset L^2(J)$ be of Type 2.
	Then there exists a constant $\gamma_0^\star > 0$ 
	independent of $V_h$, $E$ and $F$,
	such that 
	the discrete inf-sup condition \eqref{e:is0} holds 
	for the discrete trial and test spaces \eqref{e:XhYh}
	with
	\begin{align}
	\label{e:p:T2:is}
		\gamma_h \geq \gamma_0^\star \kappa_h
		.
	\end{align}
\end{proposition}

We remark that $\gamma_0^\star$ in \eqref{e:p:T2:is},
as a function of 
the number of refinements between $\Tcal_E$ and $\Tcal_F$,
is monotonically increasing
and saturates (exponentially quickly).

Type 1 temporal subspaces,
on the other hand,
do not lead to unconditional stability
of the form \eqref{e:p:T2:is}.
To formalize this,
we define the CFL number 
\begin{align}
	\mathrm{CFL}_h
	:=
	\max\Delta \Tcal_E 
	\sup_{ \chi_h \in V_h \NZ }
	\frac{ \norm{ \chi_h }{V} }{ \norm{ \chi_h }{V'} }
	.
\end{align}

\begin{proposition}
\label{p:T1}
	Let $\set{0} \neq V_h \subset V$ be a finite-dimensional subspace.
	Let $E \subset H^1(J)$ and $F \subset L^2(J)$ be of Type 1.
	Then there exists a constant $\gamma_0 > 0$ 
	independent of $V_h$, $E$ and $F$,
	such that 
	the discrete inf-sup condition \eqref{e:is0} holds
	for the discrete trial and test spaces \eqref{e:XhYh}
	with
	\begin{align}
	\label{e:p:T1:is}
		\gamma_h \geq \gamma_0 \kappa_h \min\set{ 1, \mathrm{CFL}_h^{-1} }
		.
	\end{align}
	In general, 
	the dependence on 
	the CFL number cannot be improved.
\end{proposition}

\section{Minimal residual Petrov-Galerkin discretization}
\label{s:mrpg}

In this section we consider 
an abstract continuous bilinear form $B : X \times Y \to \IR$,
where $X$ and $Y$ are Hilbert spaces
with norms $\norm{\cdot}{X}$ and $\norm{\cdot}{Y}$.
Let
\[
	\norm{B}{}
	:=
	\sup_{ w \in X \NZ }
	\sup_{ v \in Y \NZ }
	\frac{ |B(w, v)| }{ \norm{w}{X} \norm{v}{Y} }
\]
denote the norm of the bilinear form $B$.
Further,
let
$b$ be a linear continuous functional of $Y$.
For the remainder of the section,
two finite-dimensional subspaces
$X_h \subset X$ and $Y_h \subset Y$
are fixed.
We aim at relaxing
the discrete variational formulation
\eqref{e:uh0}
to admit the case $\dim X_h < \dim Y_h$.
To guarantee well-posedness,
the discrete inf-sup condition of $B$ on $X_h \times Y_h$
will be essential (\cf \Pro \ref{p:T2}):
\begin{align}
\label{e:is}
	\gamma_h :=
	\inf_{w_h \in X_h \NZ}
	\sup_{v_h \in Y_h \NZ}
	\frac{
		B(w_h, v_h)
	}{
		\norm{ w_h }{X}
		\norm{ v_h }{Y}
	}
	> 0
	.
\end{align}

We introduce
norms $\tnorm{\cdot}{X}$ and $\tnorm{\cdot}{Y}$
on $X_h$ and $Y_h$
that are
induced by
(positive definite)
linear continuous operators
$M : X_h \to X'$ and $N : Y_h \to Y'$
as follows:
\begin{align}
	\tnorm{ w_h }{X}^2 & := (M w_h)(w_h),
	\quad w_h \in X_h,
\\
	\tnorm{ v_h }{Y}^2 & := (N v_h)(v_h),
	\quad v_h \in Y_h
	.
\end{align}
The operators $M$ and $N$ are 
moreover assumed to be symmetric,
\ie
$(M u_h)(w_h) = (M w_h)(u_h)$ for all $w_h, u_h \in X_h$,
and similarly for $N$.
The Gram matrices
of the operators $M$ and $N$, defined below,
will essentially act as preconditioners for the discrete system,
and should therefore be easy to invert approximately,
cf.~\Sec \ref{s:i:inv}.
Let $0 < d_M \leq D_M < \infty$ and $0 < d_N \leq D_N < \infty$
be constants such that
\begin{align}
\label{e:cC}
	d_M \tnorm{\cdot}{X} \leq \norm{\cdot}{X} \leq D_M \tnorm{\cdot}{X}
	\TEXT{and}
	d_N \tnorm{\cdot}{Y} \leq \norm{\cdot}{Y} \leq D_N \tnorm{\cdot}{Y}
\end{align}
on $X_h$ and $Y_h$, respectively.
We emphasize that
the operators $M$ and $N$, the induced norms,
and hence the constants in \eqref{e:cC}
may depend on $h$,
but in this section,
the pair
$X_h \times Y_h$ is fixed
to lighten the notation.

Instead of the usual 
discrete variational formulation
we now introduce
the discrete (functional) residual minimization problem
\begin{align}
\label{e:uh}
	u_h := \argmin_{ w_h \in X_h } {R_h(w_h)},
\end{align}
where
\begin{align}
	{R_h(w_h)} := 
	\sup_{ v_h \in Y_h \NZ }
	\frac{ | B(w_h, v_h) - b(v_h) | }{ \tnorm{v_h}{Y} },
	\quad w_h \in X_h,
\end{align}
is the (functional) residual.
The following can be shown
\cite{Andreev2013}.

\begin{theorem}
	Let the discrete inf-sup condition \eqref{e:is} hold.
	Then there exists a unique $u_h \in X_h$
	for which 
	\begin{align}
		R_h(u_h) \leq R_h(w_h)
		\quad \forall w_h \in X_h
	\end{align}
	holds.
	Moreover,
	$u_h$ satisfies the quasi-optimality estimate
	\begin{align}
	\label{e:qo}
		\norm{ u - u_h }{X} 
		\leq
		C_h \inf_{ w_h \in X_h } \norm{ u - w_h }{X}
		\TEXT{with}
		C_h =
		\frac{ \norm{ B }{} }{ \gamma_h }
		\frac{ D_N }{ d_N }
	\end{align}
	for any $u \in X$
	such that
	$B(u, v) = b(v)$ for all $v \in Y$.
\end{theorem}

\begin{proof}[Sketch]
	Invoking the open mapping theorem,
	one can show that 
	the 
	the mapping $b \mapsto u_h$
	is well-defined, linear and continuous,
	and
	its norm is dominated by
	$\frac{1}{\gamma_h} \frac{D_N}{d_N}$.
	Thus
	the composition
	$u \mapsto B u \mapsto u_h$
	is a continuous projection
	with norm not exceeding
	$C_h$ given in \eqref{e:qo}.
	An application of 
	\cite[\Lem 5]{ZikatanovXu2003}
	finishes the argument.
\end{proof}

Let us describe a computable algebraic equivalent
of
the somewhat nonstandard
variational definition \eqref{e:uh}
of the discrete solution.
To that end
let
$\Phi \subset X_h$ and $\Psi \subset Y_h$
be bases for the respective discrete spaces.
Having fixed the pair $X_h \times Y_h$,
the possible dependence on $h$ is again omitted.
The algebraic representants
of the system bilinear form $B$,
of the load functional $b$
and
of
the norm-inducing operators $N$ and $M$,
are defined \wrt the chosen basis 
in the usual way,
\begin{align}
	\Bbf := B(\Phi, \Psi),
	\quad
	\bbf := b(\Psi),
	\quad
	\Nbf := (N \Psi)(\Psi),
	\quad
	\Mbf := (M \Phi)(\Phi),
\end{align}
or in componentwise notation
$\Bbf_{\psi \phi} = B(\phi, \psi)$,
$\bbf_{\psi} = b(\psi)$,
$\Nbf_{\psi' \psi} = (N \psi)(\psi')$,
$\Mbf_{\phi' \phi} = (M \phi)(\phi')$
for $\phi, \phi' \in \Phi$ and $\psi, \psi' \in \Psi$.
The basis functions
are used
to index the components of matrices and vectors.
Similarly,
$\IR^\Phi$
will denote vectors of real numbers indexed by $\phi \in \Phi$.
The matrix $\Bbf$ is injective if and only if
the discrete inf-sup condition \eqref{e:is} holds;
further, $\Nbf$ and $\Mbf$ are symmetric positive definite matrices
due to the analogous properties \eqref{e:cC}
of the operators $N$ and $M$.
Thus,
$\norm{ \wbf }{ \Mbf } := \sqrt{ \wbf^\T \Mbf \wbf }$, $\wbf \in \IR^\Phi$,
defines a norm,
and we use similar notation for other matrices.

With these definitions,
the discrete functional residual minimization \eqref{e:uh}
can be seen to be equivalent
to the discrete algebraic residual minimization 
\begin{align}
\label{e:ubf}
	\ubf := 
	\argmin_{\wbf \in \IR^\Phi} 
	\norm{ \Bbf \wbf - \bbf }{ \Nbf^{-1} }
	.
\end{align}
A vector $\ubf$
is a stationary point of \eqref{e:ubf}
if and only if
it satisfies
the first order optimality conditions,
namely the generalized Gau{\ss} normal equations
\begin{align}
\label{e:GNE}
	\Bbf^\T \Nbf^{-1} \Bbf \ubf
	=
	\Bbf^\T \Nbf^{-1} \bbf
	.
\end{align}

If the discrete inf-sup condition \eqref{e:is} holds,
the matrix $\Bbf^\T \Nbf^{-1} \Bbf$ is 
symmetric positive definite;
then,
the Gau{\ss} normal equations \eqref{e:GNE},
and therefore also the discrete minimization problem \eqref{e:ubf},
have a unique solution.
Finally, 
the matrices $\Mbf$ and $\Bbf^\T \Nbf^{-1} \Bbf$ 
are spectrally equivalent
with the bounds
\cite[\Sec 4.1]{AndreevDiss}
\begin{align}
\label{e:M-bdd}
	\gamma_h d_M d_N
	\norm{ \wbf }{\Mbf}
	\leq
	\norm{ \wbf }{\Bbf^\T \Nbf^{-1} \Bbf}
	\leq
	\norm{B}{} D_M D_N
	\norm{ \wbf }{\Mbf}
	\quad
	\forall \wbf \in \IR^\Phi
	.
\end{align}
Therefore,
$\Mbf$ is a preconditioner for 
the Gau{\ss} normal equations \eqref{e:GNE}.
%
%
By the estimate \eqref{e:M-bdd},
the quality of this preconditioner
is controlled by
the discrete inf-sup constant $\gamma_h$ in \eqref{e:is}
and the norm equivalence constants in \eqref{e:cC},
and
does not depend on the choice of the basis.

\section{Parabolic space-time preconditioners}
\label{s:mn}

In \Sec \ref{s:mrpg}
we admitted general
norm-inducing operators $M$ and $N$
on the fixed pair of
finite-dimensional discrete trial and test spaces $X_h \times Y_h$.
For the space-time variational formulation \eqref{e:vf}
several practical choice are available
\cite[\Cha 6]{AndreevDiss}.
Here, to simplify the exposition,
we will only use the canonical choice
of the Riesz mappings
defined (on all of $X$ and $Y$) by
\begin{align}
\label{e:M}
	(M w)(w) & := 
	\norm{ w }{X}^2 = 
	\norm{ w }{L^2(J; V)}^2 + \norm{ \Dt w }{L^2(J; V')}^2,
	\quad w \in X,
\intertext{and}
\label{e:N}
	(N v)(v) & :=
	\norm{ v }{Y}^2 = 
	\norm{ v_1 }{L^2(J; V)}^2 + \norm{ v_2 }{H}^2,
	\quad v = (v_1, v_2) \in Y
	.
\end{align}
These definitions extend to the off-diagonal
by the imposed symmetry of $M$ and $N$.
In these formulas,
the $V = H_{\Gamma_0}^1(D)$ and the $V' = [H_{\Gamma_0}^1(D)]'$ norms
are taken to be the ``energy norms'':
\begin{align}
\label{e:M:V}
	\norm{ \sigma }{V}^2
	& :=
	\scalar{ a \GRAD \sigma, \GRAD \sigma }{D},
	\quad \sigma \in V,
	\\
\label{e:M:V'}
	\norm{ \varphi }{V'}^2
	& :=
	\scalar{ \varphi, A^{-1} \iota \varphi }{D},
	\quad
	\varphi \in H = L^2(D).
\end{align}
Here, $A$ is the operator 
$A : V \to V'$, $\sigma \mapsto \scalar{a \GRAD \sigma, \GRAD \cdot}{D}$,
where $a$
is
the heat conduction coefficient from \eqref{e:pde:u}
satisfying the bounds \eqref{e:a},
and
$\iota \varphi \in V'$ is the functional on $V$
defined by 
$\iota \varphi := \scalar{\varphi, \cdot}{D}$, $\varphi \in H$.
For $\beta \in V'$,
the definition extends
to $\norm{ \beta }{V'}^2 := \beta(A^{-1} \beta)$.
%

\section{Kronecker product structure of the discretized operators}
\label{s:matrix}

\subsection{The system bilinear form}
\label{s:matrix:B}

Recall from \Sec \ref{s:model} the definition
of the system bilinear form
\begin{align*}
	B(u, v) :=
	\int_J \scalar{\Dt u, v_1}{D} dt
	+
	\int_J \scalar{a \GRAD u, \GRAD v_1}{D} dt
	+
	\scalar{u(0, \cdot), v_2}{D}
\end{align*}
for the space-time variational formulation
of the model parabolic evolution equation,
where 
$u \in X = L^2(J; V) \cap H^1(J; V')$
and
$v = (v_1, v_2) \in Y = L^2(J; V) \times H$.
As described in \Sec \ref{s:sttp},
we consider two types of 
discrete trial and test spaces.
In either case these have the form
\begin{align}
	X_h = E \tensor V_h \subset X
	\TEXT{and}
	Y_h = (F \tensor V_h) \times V_h \subset Y,
\end{align}
where
	$E \subset H^1(J)$ is the space of
	continuous piecewise affine functions on 
	a temporal mesh,
	$F \subset L^2(J)$ is the space of piecewise constant functions
	on the same mesh (Type 1)
	or
	on its $n$-fold uniform refinement (Type 2),
	and
	$V_h \subset V$ is a finite-dimensional subspace.

For the remainder of the section
we fix
the spatial discretization $V_h \subset V$
with a basis $\Sigma \subset V_h$,
and
the temporal mesh
$\Tcal_E = \{ 0 = t_0 < t_1 < \ldots < t_K = T \}$.
Let 
$\Tcal_F = \{ 0 = t_0' < t_1' < \ldots < t_{K'}' = T \}$
be either $\Tcal_E$ 
or 
any $n$-fold uniform refinement of $\Tcal_E$.
Let $E$ be as above,
and let $F$ 
denote the space
of piecewise constant functions \wrt 
the temporal mesh $\Tcal_F \supseteq \Tcal_E$,
which possibly refines $\Tcal_E$.
As basis for $E$
we take the usual hat functions
$\Theta := \set{ \theta_k : k = 0, \ldots, K } \subset E$
defined by $\theta_k(t_{\tilde{k}}) = \delta_{k \tilde{k}}$,
where $\delta_{k \tilde{k}}$ denotes the Kronecker delta.
In particular, the only function
that does not vanish at $t = 0$
is $\theta_0$.
As basis for $F$
we take the indicator functions
$\Xi := \set{ \xi_k := \chi_{(t_{k-1}', t_{k}')} : k = 1, \ldots, K' }$
on the elements $(t_{k-1}', t_{k}')$ of
the temporal mesh $\Tcal_F \supseteq \Tcal_E$.
These univariate bases
are first combined to the collections
$\Phi \subset X$
and
$\Psi_1 \subset Y_1$
as
\begin{align}
	\Phi := 
	\set{ \theta \tensor \sigma : \theta \in \Theta, \sigma \in \Sigma },
	\quad
	\Psi_1 := 
	\set{ \xi \tensor \sigma : \xi \in \Xi, \sigma \in \Sigma }
	.
\end{align}
These now yield bases 
\begin{align}
	\Phi \subset X_h
	\TEXT{and}
	\Psi :=
	(\Psi_1 \times \set{0}) \cup (\set{0} \times \Sigma) \subset Y_h
\end{align}
for the discrete trial and test spaces $X_h$ and $Y_h$.
Discretizing the bilinear form $B$
using these tensor product bases $\Phi$ and $\Psi$
as described in abstract terms in \Sec \ref{s:mrpg}
leads to
\begin{align}
	\Bbf = 
	\begin{pmatrix}
		\Cbf_t^{FE} \tensor \Mbf_x + \Mbf_t^{FE} \tensor \Abf_x
		\\
		\ebf_t^{E} \tensor \Mbf_x
	\end{pmatrix}
\end{align}
where
\textbf{a)}
	the ``temporal FEM'' matrices
	$\Cbf_t^{FE}, \Mbf_t^{FE} \in \IR^{\Xi \times \Theta}$
	and
	the row vector
	$\ebf_t^{E} \in \IR^{\Theta}$,
	have the components
	\begin{align}
		[\Cbf_t^{FE}]_{\xi \theta}
		=
		\int_J \theta'(t) \xi(t) dt,
		\quad
		[\Mbf_t^{FE}]_{\xi \theta}
		=
		\int_J \theta(t) \xi(t) dt,
		\quad
		[\ebf_t^{E}]_{\theta} 
		=
		\delta_{\theta_0 \theta}
		,
	\end{align}
	with the prime denoting the derivative \wrt $t$,
	and
\textbf{b)}
	the usual ``spatial FEM''
	mass and stiffness matrices
	$\Mbf_x, \Abf_x \in \IR^{\Sigma \times \Sigma}$
	are given by
	\begin{align}
		[\Mbf_x]_{\tilde\sigma \sigma}
		=
		\int_D \tilde\sigma(x) \sigma(x) dx,
		\quad
		[\Abf_x]_{\tilde\sigma \sigma}
		=
		\int_D a(x) \GRAD \tilde\sigma(x) \cdot \GRAD \sigma(x) dx
		.
	\end{align}

Let us comment 
on the assembly of the temporal FEM matrices.
First, if $\Tcal_E = \Tcal_F$ (Type 1)
then
$[\Cbf_t^{FE}]_{\xi \theta} \in \{ 1, -1, 0 \}$
and
$[\Mbf_t^{FE}]_{\xi \theta} \in \{ \frac12 |I|, 0 \}$
depending on whether $\xi$ and $\theta$ are both nonzero on 
the temporal element $I$ of $\Tcal_E$ having length $|I|$,
and on the sign of $\theta'$ there.
Therefore,
assume now that 
$\Tcal_F$ is obtained from $\Tcal_E$
by a succession of uniform refinements (Type 2).
Let $\Tcal_E^\star$ denote the first uniform refinement of $\Tcal_E$,
and let
$E^\star$ 
be the space of continuous piecewise affine function on $\Tcal_E^\star$
with the hat function basis $\Theta^\star$.
Let
$\Cbf_t^{F E^\star}$
and
$\Mbf_t^{F E^\star}$
denote the matrices as above, 
but with $\Tcal_E^\star$ in place of $\Tcal_E$.
Consider now the embedding operator $S^{E} : E \to E^\star$.
Define the prolongation matrix
$[\Sbf_t^{E}]_{\theta^\star \theta}$
by
$
	S^{E} \theta = 
	\sum_{\theta^\star \in \Theta^\star } 
	\Sbf_{\theta^\star \theta}^{E}
	\theta^\star
$,
$\theta \in \Theta$.
Then
\begin{align}
\label{e:MMS}
	\Cbf_t^{F E} = \Cbf_t^{F E^\star} \Sbf_t^{E}
	\TEXT{and}
	\Mbf_t^{F E} = \Mbf_t^{F E^\star} \Sbf_t^{E}
	.
\end{align}
Moreover,
denoting by $t_\theta \in \Tcal$ 
the node for which $\theta(t_\theta) = 1$,
and similarly for $t_{\theta^\star} \in \Tcal^\star$,
we have
``$\Tcal_E^\star = \Sbf_t^E \Tcal_E$'', \ie
\begin{align}
\label{e:TTS}
	t_{\theta^\star} = 
	\sum_{\theta \in \Theta} 
	[\Sbf_t^{E}]_{\theta^\star \theta}
	t_{\theta}
	\quad
	\forall
	\theta^\star \in \Theta^\star
	.
\end{align}

\subsection{The norm-inducing operators}

With the norms on $V$ and $V'$ taken to be \eqref{e:M:V}--\eqref{e:M:V'},
the discretized operators $M$ and $N$ defined
in \Sec \ref{s:mn}
assume the form
\begin{align}
\label{e:Mbf}
	\Mbf =
	\Mbf_t^{E} \tensor \Abf_x + \Abf_t^{E} \tensor (\Mbf_x \Abf_x^{-1} \Mbf_x)
\end{align}
with
the ``temporal mass and stiffness'' matrices
\begin{align}
	[\Mbf_t^{E}]_{\tilde\theta\theta} = 
	\int_J \tilde\theta(t) \theta(t) dt,
	\quad
	[\Abf_t^{E}]_{\tilde\theta\theta} =
	\int_J \tilde\theta'(t) \theta'(t) dt,
\end{align}
and
the block-diagonal matrix
\begin{align}
\label{e:Nbf}
	\Nbf = 
	\begin{pmatrix}
		\Mbf_t^{F} \tensor \Abf_x
		& \mathbf{0} \\ \mathbf{0} &
		\Mbf_x
	\end{pmatrix}
\end{align}
with the ``temporal mass matrix''
\begin{align}
	[\Mbf_t^{F}]_{\tilde\xi\xi} = 
	\int_J \tilde\xi(t) \xi(t) dt
	.
\end{align}

\subsection{Data structures}

The Kronecker product structure of these matrices
suggests regarding
a vector $\wbf \in \IR^{\Sigma \times \Theta}$
as a rectangular array 
with $\#\Sigma$ rows and $\#\Theta$ columns.
Let $\mathop{\mathrm{Vec}}(\wbf)$ denote
the ``vectorization'' of such an array,
\ie its columns are collected one after another into one column vector 
$\mathop{\mathrm{Vec}}(\wbf)$ 
of length $\#(\Sigma \times \Theta)$.
Now, 
if
$\Tbf \in \IR^{\Theta \times \Theta}$
and
$\Xbf \in \IR^{\Sigma \times \Sigma}$
are matrices
then
\begin{align}
\label{e:Vec}
	(\Tbf \tensor \Xbf) 
	\mathop{\mathrm{Vec}}(\wbf)
	=
	\mathop{\mathrm{Vec}}( \Xbf \wbf \Tbf^\T )
	.
\end{align}
In the implementation
we will exclusively use
the representation as rectangular arrays.
Moreover,
load vectors derived from load functionals $d \in Y'$
will be 
stored as pairs
$\dbf = (\dbf^1, \dbf^2)$
with
$\dbf^1 \in \IR^{\Sigma \times \Xi}$
(rectangular array with $\#\Sigma$ rows and $\#\Xi$ columns)
and
$\dbf^2 \in \IR^{\Sigma}$
(column vector of length $\#\Sigma$)
in the form of a Matlab structure \lstinline|{d1, d2}|.
To these,
a formula analogous to \eqref{e:Vec} applies.
In particular, 
we never store the operators 
$\Bbf$, $\Mbf$ and $\Nbf$ (or its inverses)
as matrices.

\section{Implementational aspects} \label{s:i}

\subsection{Inverses of the space-time parabolic preconditioners} \label{s:i:inv}

Consider 
the matrix representations
$\Mbf$ and $\Nbf$
of
the space-time parabolic preconditioners 
given in \eqref{e:Mbf} and \eqref{e:Nbf} in \Sec \ref{s:matrix}.
In order to solve 
the generalized Gau{\ss} normal equations \eqref{e:GNE}
with $\Mbf$ as preconditioner,
we need to (approximately) compute
the inverses $\Mbf^{-1}$ and $\Nbf^{-1}$.

\subsubsection{The test side} \label{s:i:inv:n}

For $\Nbf$ we simply use
the block-diagonal representation
\begin{align}
	\Nbf^{-1} = 
	\begin{pmatrix}
		(\Mbf_t^{F})^{-1} \tensor \Abf_x^{-1}
		& \mathbf{0} \\ \mathbf{0} & 
		\Mbf_x^{-1}
	\end{pmatrix}
	,
\end{align}
which again has Kronecker product structure.
For problems with a large number of spatial degrees of freedom,
the inverse $\Abf_x^{-1}$ may replaced by an approximate inverse,
\eg several cycles of a multigrid method.

\subsubsection{The trial side} \label{s:i:inv:m}

The representation of $\Mbf^{-1}$ requires 
more discussion,
as it is not (obviously) of Kronecker product structure.
We will obtain a simplified expression for $\Mbf^{-1}$
by diagonalizing $\Mbf_t^{E}$ and $\Abf_t^{E}$.
For this discussion, let us drop the superscript $(\cdot)^{E}$.
Consider therefore the generalized eigenvalue problem
of finding $\vbf \in \IR^{\Theta}$ and $\lambda \in \IR$
such that
$\Abf_t \vbf = \lambda \Mbf_t \vbf$
(in place of $\Mbf_t$,
one could use the mass-lumped version of $\Mbf_t$,
or simply the diagonal matrix that coincides 
with $\Mbf_t$ on the diagonal).
Let $\Ibf_t$ denote the identity matrix
of the same size as $\Mbf_t$ and $\Abf_t$.
Since $\Abf_t$ is symmetric positive semi-definite
and $\Mbf_t$ is symmetric positive definite,
all eigenvalues are nonnegative
and the eigenvectors may be chosen to form an $\Mbf_t$-orthogonal basis:
there exists a square matrix $\Vbf_t$ 
collecting the eigenvectors in its columns,
and a diagonal matrix $\Dbf_t$
containing the eigenvalues on the diagonal,
such that
\begin{align}
	\Vbf_t^\T \Mbf_t \Vbf_t = \Ibf_t
	\TEXT{and}
	\Abf_t \Vbf_t = \Mbf_t \Vbf_t \Dbf_t
	.
\end{align}
Let us set $\Tbf_t := \Mbf_t \Vbf_t$.
The first identity implies 
$\Vbf_t^{-1} = \Vbf_t^\T \Mbf_t = \Tbf_t^\T$,
which may be used to verify
\begin{align}
	\Mbf_t = \Tbf_t \Tbf_t^\T
	\TEXT{and}
	\Abf_t = \Tbf_t \Dbf_t \Tbf_t^\T
	.
\end{align}
Inserting these into the expression \eqref{e:Mbf} for $\Mbf$
we find
\begin{align}
	\Mbf = 
	(\Tbf_t \tensor \Ibf_x) 
	(\Ibf_t \tensor \Abf_x + \Dbf_t \tensor (\Mbf_x \Abf_x^{-1} \Mbf_x)) 
	(\Tbf_t^\T \tensor \Ibf_x)
\end{align}
and therefore
\begin{align}
	\Mbf^{-1} = 
	(\Vbf_t \tensor \Ibf_x) 
	(\Ibf_t \tensor \Abf_x + \Dbf_t \tensor (\Mbf_x \Abf_x^{-1} \Mbf_x))^{-1}
	(\Vbf_t^\T \tensor \Ibf_x)
	.
\end{align}
Let $\gamma_\theta$ be
the square root of the entry of $\Dbf_t$ on the diagonal
in position $\theta$,
\ie
$\Dbf_t = \mathop{\mathrm{diag}} ( (\gamma_\theta^2)_{\theta \in \Theta} )$.
Now, 
recall from \Sec \ref{s:matrix} the convention
that $\wbf \in \IR^{\Sigma \times \Theta}$
is stored as a rectangular array
with $\#\Theta$ columns.
Applying
$\Mbf^{-1}$
to such a vector
$\wbf$
will be done as follows
\begin{enumerate}
\item 
	Compute $\wbf^{(1)} := \wbf \Vbf_t$.
\item
	For each column $\wbf_\theta^{(1)}$ of $\wbf^{(1)}$
	compute the column $\wbf_\theta^{(2)}$ of $\wbf^{(2)}$
	by
	\begin{align}
	\label{e:CoreM}
		\wbf_\theta^{(2)}
		:=
		(\Abf_x + \gamma_\theta^2 (\Mbf_x \Abf_x^{-1} \Mbf_x))^{-1}
		\wbf_\theta^{(1)}
		.
	\end{align}
\item
	Compute $\wbf^{(3)} := \wbf^{(2)} \Vbf_t^\T$.
	Then $\Mbf^{-1} \wbf = \wbf^{(3)}$.
\end{enumerate}

The computation \eqref{e:CoreM}
can be done in parallel over the columns.
We will further make use of the following identity,
valid when applied to a real vector as in \eqref{e:CoreM},
\begin{align}
\label{e:Helmholtz}
	(\Abf_x + \gamma_\theta^2 (\Mbf_x \Abf_x^{-1} \Mbf_x))^{-1}
	=
	\mathop{\mathrm{Re}} 
	\mathbin\circ
	( \Abf_x + i \gamma_\theta \Mbf_x )^{-1}
	.
\end{align}
The right-hand-side
means
the solution of
the FEM discretization
of the Helmholtz operator
$(A + i \gamma_\theta \mathop{\mathrm{Id}})$
with imaginary frequency $i \gamma_\theta$,
followed by taking the real part.
Interestingly,
such Helmholtz problems appeared
in the context of
(parabolic) evolution equations
in \cite{SheenSloanThomee2003,BanjaiPeterseim2012},
but in the present case only in
the representation of the preconditioner $\Mbf^{-1}$.
These may therefore be inverted approximately.
%

\subsection{Assembly of the space-time load vector} \label{s:i:b}

As in the previous section,
let $F$ be the space of piecewise constant functions 
on a temporal mesh $\Tcal_F$;
let $\Xi$ be the basis on $F$
consisting of indicator functions
on the elements of $\Tcal_F$;
finally, $\Sigma \subset V_h$ is a basis for 
a finite-dimensional subspaces $V_h \subset V = H_{\Gamma_0}^1(D)$.
The load functional $b(v_1, v_2)$ defined in \Sec \ref{s:model} 
can be rewritten as
\begin{align}
	b(v_1, 0) + b(0, v_2) =
	\int_J
	\Set{ \scalar{f, v_1}{D} + \scalar{g, v_1}{\Gamma_N} } 
	dt
	+
	\scalar{ h, v_2 }{D},
\end{align}
for $(v_1, v_2) \in Y = L^2(J; V) \times H$,
where
$h \in H = L^2(D)$,
$f \in L^2(J; H)$ and $g \in L^2(J; H^{1/2}(\Gamma_N))$
are given functions.
Accordingly,
the load vector $\bbf$
that we obtain as outlined in \Sec \ref{s:mrpg}
consists of two parts,
say,
$\bbf^1 \in \IR^{\Sigma \times \Xi}$ and $\bbf^2 \in \IR^{\Sigma}$.
The second part is given 
in componentwise notation by
\begin{align}
	[\bbf^2]_{\sigma} = \scalar{h, \sigma}{D} + \scalar{0, \sigma}{\Gamma_N},
	\quad
	\sigma \in \Sigma
	,
\end{align}
which is the usual spatial FEM load vector for 
the function $h \in L^2(D)$
with zero Neumann data.
In the remainder of this section
we describe 
how $\bbf^1$ may be obtained
from the usual spatial FEM load vectors.
Observe first 
that,
fixing an indicator function $\xi \in \Xi$
supported on the closed interval $I$,
we may employ a quadrature rule on $I$ to define
\begin{align}
	[\bbf^1]_{\sigma \xi}
	& :=
	\sum_{r \in \IN}
	w_r^I
	\Set{ \scalar{f(t_r^I,\cdot), \sigma}{D} + \scalar{g(t_r^I,\cdot), \sigma}{\Gamma_N} } 
	\\
	& \approx
	\int_{I} 
	\Set{ \scalar{f, v_1}{D} + \scalar{g, v_1}{\Gamma_N} } 
	dt
	,
	\quad
	\sigma \in \Sigma,
\end{align}
where $t_r^I \in I$ are the quadrature nodes
and $w_r^I \in \IR$ are the quadrature weights
(equal to zero for $r \in \IN$ large enough).
Now, each term
in the curly brackets $\Set{ \ldots }$
is a load vector for 
the function $f(t_r^I,\cdot)$
with Neumann data $g(t_r^I,\cdot)$,
which can be assembled using
standard spatial FEM routines.
This presupposes sufficient regularity 
of the functions $t \mapsto f(t, \cdot)$ and $t \mapsto g(t, \cdot)$
on the interval $I$.
The computation of individual columns of $\bbf$
(\ie for each given $\xi \in \Xi$)
can be performed in parallel.


If $\Tcal_F = \Tcal_E$
then using the trapezoidal rule on each temporal element $I$
leads to a system
$\Bbf \ubf = \bbf$
that
admits a unique solution,
which on the nodes of $\Tcal_E$
coincides
with
the solution
obtained by
the Crank-Nicolson time-stepping method
\cite{Hulme1972-118,AkrivisMakridakisNochetto2011}.

\subsection{Generalized LSQR algorithm} \label{s:i:glsqr}

In \Sec \ref{s:mrpg}
the minimal residual Petrov-Galerkin discretization
was shown to lead to 
a system of generalized Gau{\ss} normal equations.
%
%
One option to solve the system is
the LSQR algorithm of
Paige and Saunders \cite{PaigeSaunders1982}
applied directly to 
the preconditioned equation
$
	\widetilde\Bbf^\T
	\widetilde\Bbf
	\widetilde\ubf
	=
	\widetilde\Bbf^\T
	\widetilde\bbf
$
with
$\widetilde\Bbf = \Nbf^{-1/2} \Bbf \Mbf^{-1/2}$,
$\widetilde\ubf = \Mbf^{-1/2} \ubf$,
$\widetilde\bbf = \Nbf^{-1/2} \bbf$,
where $\Mbf^{-1/2}$ and $\Nbf^{-1/2}$
denote
their inverses of the square roots
of the (s.p.d.) matrices $\Mbf$ and $\Nbf$.
%
%
We reformulate
the algorithm
it in such a way that only
the inverses $\Mbf^{-1}$ and $\Nbf^{-1}$
need to be applied, but not the square roots,
\cf \cite{Benbow1999}.
To compute
an approximate
solution ${\ubf}_{i^\star} \approx {\ubf}$
to the generalized Gau{\ss} normal equations \eqref{e:GNE}
using $\Mbf$ as a preconditioner,
the algorithm proceeds as follows:
%
\begin{enumerate}
	\item
		Initialize
		\begin{enumerate}
		\item
			$\dbf_0 := \zerobf$
		\item
			$(\widehat{\vbf}_0, {\vbf}_0, \beta_0) := \textsc{Normalize}(\bbf, \Nbf)$
		\item
			$(\widehat{\wbf}_0, {\wbf}_0, \alpha_0) := \textsc{Normalize}(\Bbf^\T \widehat{\vbf}_0, \Mbf)$
		\item
			$\rho_0 := \norm{(\alpha_0, \beta_0)}{2}$
		\item
			${\ubf}_0 := \zerobf$
		\item
			$\delta_1 = \alpha_0$, $\gamma_1 = \beta_0$
		\end{enumerate}
	\item
		For $i = 1, 2, \ldots, i^\star$ do the following steps (until convergence)
		\begin{enumerate}
		\item
			$\dbf_{i} := \widehat{\wbf}_{i-1} - (\alpha_{i-1} \beta_{i-1} / \rho_{i-1}^2) \dbf_{i-1}$
		\item
			$
				(\widehat{\vbf}_{i}, {\vbf}_{i}, \beta_{i})
				:=
				\textsc{Normalize}(\Bbf \widehat{\wbf}_{i-1} - \alpha_{i-1} {\vbf}_{i-1}, \Nbf)
			$
		\item
			$
				(\widehat{\wbf}_{i}, {\wbf}_{i}, \alpha_{i})
				:=
				\textsc{Normalize}(\Bbf^\T \widehat{\vbf}_{i} - \beta_{i} {\wbf}_{i-1}, \Mbf)
			$
		\item
			$\rho_i := \norm{(\delta_i, \beta_{i})}{2}$,
		\item
			${\ubf}_i := {\ubf}_{i-1} + (\delta_i \gamma_i / \rho_i^2) \dbf_i$
		\item
			$\delta_{i+1} := -\delta_i \alpha_{i} / \rho_i$,
			$\gamma_{i+1} := \gamma_i \beta_{i} / \rho_i $
		\end{enumerate}
\end{enumerate}
Here,
$
	\textsc{Normalize} :
	(\sbf, \Sbf) \mapsto (\widehat{\zbf}, {\zbf}, z)
$,
with $\Sbf$ s.p.d.,
is the procedure:
\begin{enumerate}
	\item
	Solve $\Sbf \widehat{\sbf} =  \sbf$ for $\widehat{\sbf}$
	\item
	Set $z := \sqrt{ \sbf^\T \widehat{\sbf} }$
	and $(\widehat{\zbf}, {\zbf}) := (z^{-1} \widehat{\sbf}, z^{-1} \sbf)$
\end{enumerate}

As long as the order of the statements is unchanged,
the subscripts $(\cdot)_i$, etc.,
can be ignored in the implementation.
In our implementation
we will limit the number of iterations
and
allow the iteration to exit
when
the normal equations residual 
$
	\norm{ 
	\widetilde\Bbf^\T
	\widetilde\Bbf
	\widetilde\ubf
	-
	\widetilde\Bbf^\T
	\widetilde\bbf }{2}
	=
	\norm{ \Bbf^\T \Nbf^{-1} \Bbf \ubf_i - \Bbf^\T \Nbf^{-1} \bbf }{ \Mbf^{-1} }
$
falls below a threshold.
This residual is available
for each $i = 0, 1, \ldots$
as $|\delta_{i+1}| \gamma_{i+1}$
following step (f).
See \cite{ChangPaigeTitleyPeloquin2009}
for further discussion of
stopping criteria for the LSQR algorithm.

\section{Overview of the Matlab code} \label{s:code}

In the code,
the naming convention
parallels
that of \Secs \ref{s:matrix} and \ref{s:i}.
Thus, the mesh $\Tcal_F$ is called \lstinline|TF|,
the matrix $\Mbf_t^{FE}$ is called \lstinline|MtFE|, 
the vector $\ubf$ is called \lstinline|u|,
and so on.
The subroutines
that are related to 
the temporal FEM
are prefixed with \lstinline|femT_|,
those related to spatial FEM with \lstinline|femX_|.
The only subroutine that mixes
temporal and spatial FEM
is \lstinline|femTX_assemLoad|
for the assembly of the space-time load vector $\bbf$.

\subsection{Main file} \label{s:code:main}

%
We provide 
the commented code
for the main file.
The code is embedded in a Matlab function \lstinline|spacetime|.

\begin{lstlisting}[name=main]
function spacetime
\end{lstlisting}
Initialize the spatial FEM 
(load the mesh, etc., here into global variables)
and
compute the spatial FEM mass and stiffness matrices.
The flag \lstinline|true| indicates
that this is the first-time initialization.
\begin{lstlisting}[name=main]
	femX_init(true)
	[Mx, Ax] = femX_MA();
\end{lstlisting}
Define a temporal mesh with \lstinline|K| elements
on the interval $J = (0, T)$ with $T = 20$:
\begin{lstlisting}[name=main]
	T = 20; K = 100;
	TE = T * sort([0, rand(1, K-1), 1]);
\end{lstlisting}
Determine
the number of uniform refinements
to go from $\Tcal_E$ to $\Tcal_F$,
see \Sec \ref{s:matrix}.
Setting \lstinline|use_mrpg = true| amounts
to one refinement.
\begin{lstlisting}[name=main]
	use_mrpg = true;
\end{lstlisting}
Compute the refined mesh $\Tcal_F$
and the temporal FEM matrices
as described in \Sec \ref{s:matrix}:
\begin{lstlisting}[name=main]
	[MtFE, CtFE, TF] = femT_assemFE(TE, use_mrpg);
	[MtE, AtE] = femT_assemE(TE);
	MtF = femT_assemF(TF); 
\end{lstlisting}
Define the function that computes $\wbf \mapsto \Bbf \wbf$
using the matrix representation given in \Sec \ref{s:matrix}:
\begin{lstlisting}[name=main]
	function Bw = B(w)
		Bw = { Mx * w * CtFE' + Ax * w * MtFE', Mx * w(:,1) };
	end
\end{lstlisting}
Define the function that computes $\vbf \mapsto \Bbf^\T \vbf$
using the matrix representation given in \Sec \ref{s:matrix}.
\begin{lstlisting}[name=main]
	function Btv = Bt(v)
		Btv = Mx' * v{1} * CtFE + Ax' * v{1} * MtFE;
		Btv(:,1) = Btv(:,1) + Mx' * v{2};
	end
\end{lstlisting}
Define the function that computes $\dbf \mapsto \Nbf^{-1} \dbf$
with $\Nbf^{-1}$ as in \Sec \ref{s:i:inv:n}:
\begin{lstlisting}[name=main]
	function iNd = iN(d)
		iNd = { Ax \ (d{1} / MtF), Mx \ d{2} };
	end
\end{lstlisting}
Define the function that computes $\wbf \mapsto \Mbf^{-1} \wbf$
using the algorithm 
and formula \eqref{e:Helmholtz}
given in \Sec \ref{s:i:inv:m}.
Two comments are in order.
First, 
the symmetric positive semi-definite matrix $\Abf_t^E$
is singular,
possibly leading to a small negative 
approximately computed eigenvalue.
Before taking the square root we therefore
round negative eigenvalues to zero.
Second,
the result of an application of \eqref{e:Helmholtz}
to a real vector is again a real vector,
which is enforced by taking the \lstinline|real| part.
This truncates the round-off error accumulated
in the imaginary part
and reestablishes the data type of reals.
The loop can be performed in parallel.
%
\begin{lstlisting}[name=main]
	[VtE, DtE] = eig(full(AtE), full(MtE));
	gamma = sqrt(max(0,diag(DtE)));
	function iMw = iM(w)
		iMw = w * VtE;
		parfor j = 1:length(TE)
			iMw(:,j) = real((Ax + 1i * gamma(j) * Mx) \ iMw(:,j));
		end
		iMw = iMw * VtE';
	end
\end{lstlisting}
Compute the space-time load vector $\bbf$
using a quadrature rule
according to \Sec \ref{s:i:b}.
The Matlab functions \lstinline|f|, \lstinline|g| and \lstinline|h|
are assumed to be available,
\eg as Matlab files in the same directory.
Further,
\lstinline|QR_Trapz|
is the trapezoidal quadrature rule
as described in \Sec \ref{s:code:load}.
\begin{lstlisting}[name=main]
	b = femTX_assemLoad(TF, @f, @g, @h, QR_Trapz());
\end{lstlisting}
Set the tolerance and the maximal number of iterations
for
the generalized LSQR algorithm from \Sec \ref{s:i:glsqr}
and run it
on the Gau{\ss} normal equations \eqref{e:GNE} 
with $\Mbf$ as preconditioner.
The solver may provide additional diagnostic
output parameters that are ignored here.
\begin{lstlisting}[name=main]
	tol = 1e-4; maxit = 100;
	u = glsqr(@B, @Bt, b, tol, maxit, @iM, @iN);
\end{lstlisting}
Finally, 
plot several temporal snapshots (equispaced in time)
of
the numerical solution.
These are obtained by linear interpolation
from the values at temporal nodes
to $t = 0, 1, \ldots, 5$,
and stored in the array \lstinline|U|.
\begin{lstlisting}[name=main]
	t = 0:5; U = interp1(TE, u', t)';
	for k = 1:size(U,2)
		subplot(1, size(U,2), k); femX_show(U(:,k));
	end
\end{lstlisting}
Here the Matlab function \lstinline|spacetime| ends.
\begin{lstlisting}[name=main]
end
\end{lstlisting}

\subsection{Assembly of the space-time load vector} \label{s:code:load}

The assembly of the space-time load vector $\bbf$
is performed in the Matlab function \lstinline|femTX_assemLoad|.
It receives 
the mesh $\Tcal_F$ as a (row) vector \lstinline|TF|,
as well as function handles \lstinline|f|, \lstinline|g| and \lstinline|h|.
The function handles \lstinline|f| and \lstinline|g|,
when called with one argument, say $t_r$,
are expected to return function handles
to functions 
that depend on the spatial variable only
and describe
$f(t_r, \cdot)$ and $g(t_r, \cdot)$.
The function handle \lstinline|h|
describes the initial condition $h(\cdot)$.
Finally, \lstinline|QuadRule| is a function handle
that
receives a 2-component vector describing the endpoints of an interval $I$
and
returns 
a quadrature rule on $I$
in the form
of a vector of nodes $(t_r^I)_{r = 1, \ldots, R}$
and a vector of corresponding weights $(w_r^I)_{r = 1, \ldots, R}$,
as well as the number of nodes $R$.
\begin{lstlisting}[name=load]
function b = femTX_assemLoad(TF, f, g, h, QuadRule)
\end{lstlisting}
Assemble the part $\bbf^2$ of the load vector
from the initial condition $h(\cdot)$.
\begin{lstlisting}[name=load]
	b2 = femX_b(h, @(varargin)0);
\end{lstlisting}
Assemble $\bbf^1$ by iterating over 
the temporal elements defined by $\Tcal_F$
as described in \Sec \ref{s:i:b}.
The outer loop can be computed in parallel.
\begin{lstlisting}[name=load]
	b1 = zeros(length(b2), length(TF)-1);
	parfor k = 1:size(b1,2)
		[tI, wI, R] = QuadRule(TF([k k+1]));
		for r = 1:R
			b1(:,k) = b1(:,k) + wI(r) * femX_b(f(tI(r)), g(tI(r)));
		end
	end
\end{lstlisting}
The assembled vectors $\bbf^1$ and $\bbf^2$
are combined into a Matlab structure.
\begin{lstlisting}[name=load]
	b = {b1, b2};
\end{lstlisting}
Here the Matlab function \lstinline|femTX_assemLoad| ends.
\begin{lstlisting}[name=load]
end
\end{lstlisting}

\subsection{Assembly of the temporal FEM matrices 1} \label{s:code:femT1}

Let us comment in some detail
on the computation
of the temporal FEM matrices $\Cbf_t^{FE}$ and $\Mbf_t^{FE}$
by means of 
the Matlab function \lstinline|femT_assemFE|.
It receives a temporal mesh $\Tcal_0 \supseteq \Tcal_E$
and the number \lstinline|nref| of uniform refinements
to be performed on $\Tcal_0$
to obtain $\Tcal_F$.
If \lstinline|nref| is interpreted as one or zero
if it has the logical value
\lstinline|true| or \lstinline|false|, respectively.
The output is
$\Cbf_t^{FE}$ and $\Mbf_t^{FE}$,
and the \lstinline|nref|-fold refinement of $\Tcal_0$
stored again in the variable \lstinline|T0|.
\begin{lstlisting}[name=assemFE]
function [MtFE, CtFE, T0] = femT_assemFE(T0, nref)
\end{lstlisting}
If no refinement is to be performed,
the temporal FEM matrices 
$\Cbf_t^{FE}$ and $\Mbf_t^{FE}$
can be computed directly.
Temporal meshes
are stored as row vectors.
\begin{lstlisting}[name=assemFE]
	K = length(T0);
	if (nref == 0)
		MtFE = spdiags(diff(T0)' * [1/2 1/2], 0:1, K-1, K);
		CtFE = diff(speye(K));
		return
	end
\end{lstlisting}
Otherwise,
the prolongation matrix $\Sbf_t^E$
is computed first, see \Sec \ref{s:matrix:B}.
As can be seen from \eqref{e:TTS},
it coincides with
the matrix representation
of an interpolation operator
(a more efficient but lengthier implementation
is possible here).
\begin{lstlisting}[name=assemFE]
	StE = sparse(interp1(1:K, eye(K), 1:(1/2):K));
\end{lstlisting}
Perform a uniform refinement
of the current mesh $\Tcal_0$
and
pass the new mesh recursively to \lstinline|femT_assemFE|.
Hence,
the number of refinements still to be
performed decreases by one.
We 
obtain the matrices $\Mbf_t^{F E^\star}$ and $\Cbf_t^{F E^\star}$
with respect to the meshes
$\Tcal_F$ and $\Tcal_0^\star$,
and the \lstinline|nref|-th refinement of $\Tcal_0$
(which is the desired $\Tcal_F$).
\begin{lstlisting}[name=assemFE]
	[MtFEs, CtFEs, T0] = femT_assemFE((StE * T0')', nref-1);
\end{lstlisting}
Apply the prolongation matrix
according to \eqref{e:MMS}
to obtain $\Mbf_t^{F E}$ and $\Cbf_t^{F E}$.
\begin{lstlisting}[name=assemFE]
	MtFE = MtFEs * StE; CtFE = CtFEs * StE;
\end{lstlisting}
Here the Matlab function \lstinline|femT_assemFE| ends.
\begin{lstlisting}[name=assemFE]
end
\end{lstlisting}

\subsection{Assembly of the temporal FEM matrices 2} \label{s:code:femT2}

The computation of 
the temporal stiffness and mass matrices 
$\Abf_t^E$, $\Mbf_t^E$ and $\Mbf_t^F$ is a routine task.
The Matlab code is
provided
for completeness.
\begin{lstlisting}[name=assemE]
function [MtE, AtE] = femT_assemE(TE)
	K = length(TE); h = diff(TE); g = 1./h; 
	MtE = spdiags([h 0; 0 h]' * [1 2 0; 0 2 1]/6, -1:1, K, K);
	AtE = spdiags([g 0; 0 g]' * [-1 1 0; 0 1 -1], -1:1, K, K);
end
\end{lstlisting}
\begin{lstlisting}[name=assemF]
function MtF = femT_assemF(TF)
	K = length(TF)-1;
	MtF = sparse(1:K, 1:K, abs(diff(TF)));
end
\end{lstlisting}

\section{Numerical experiments} \label{s:x}

We present two numerical experiments.
In the first,
we focus on 
the dependence of the condition number of 
the preconditioned system matrix
as a function of temporal elements.
In the second,
we focus
on the execution times.
To illustrate the modularity,
the two experiments 
are based on
two different packages for
spatial finite element discretization.
For simplicity
we set
$a = 1$ for the heat conduction coefficient
and
$f(t, x) = \sin(t)$ for the source term in \eqref{e:pde:u},
as well as
$g = 0$ for the Neumann data \eqref{e:pde:N}
and 
$h = 0$ for the initial condition \eqref{e:pde:0}.

\subsection{Dependence on the temporal resolution} \label{s:x:B}

In the first experiment
we use
the 
2d spatial FEM discretization
from \cite[\Secs 2--8]{AlbertyCarstensenFunken1999}.
The mesh consists of 6 quadrilaterals and 4 triangles,
carrying bilinear and linear finite elements, respectively,
see \cite{AlbertyCarstensenFunken1999} for details.
The code described in \Sec \ref{s:code}
(and some graphical postprocessing)
produces a figure similar to \Fig \ref{f:A}.

\begin{figure}[htbp]
	\begin{center}
		\includegraphics[width=.9\textwidth]{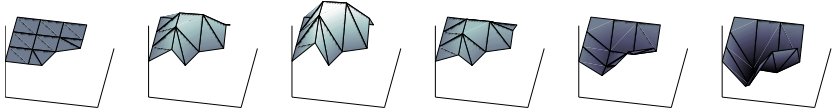}
		\caption{
			Snapshots of the solution 
			at $t = 0, 1, 2, 3, 5$,
			produced by the code given in \Sec \ref{s:code}.
			Dark corresponds to low, bright to high values.
		}
		\label{f:A}
	\end{center}
\end{figure}


In \Fig \ref{f:B}
we document
a)
the accuracy of the discrete solution with respect to the space-time norm
$\norm{\cdot}{X}$ defined in \eqref{e:M},
b)
the number of iterations of the generalized least squares solver
for
equidistant temporal meshes $\Tcal_E$ of different size
and
different numbers of refinements between $\Tcal_E$ and $\Tcal_F$,
and
c)
the condition number of the preconditioned system matrix
$\widetilde\Bbf^\T \widetilde\Bbf$, 
see \Sec \ref{s:i:glsqr}.
For measuring the accuracy of the discrete solution,
a reference solution on a fine temporal mesh is used.
We observe that
for
Type 1 temporal subspaces
(no refinement of the test space; this method is equivalent to the Crank-Nicolson time-stepping scheme),
the number of iterations 
first increases with $\# \Tcal_E$
but then decreases again.
This is explained by the increasing size of the system,
but decreasing condition number.
For Type 2 temporal subspaces 
(one or more temporal refinements of the test space),
the number of iterations
is consistently smaller.
Indeed, the condition number is approximately $2$
independently of $\# \Tcal_E$.
These observations are in agreement
with \Pros \ref{p:T2} and \ref{p:T1}.
Replacing the trapezoidal rule
by a higher-order quadrature
does not significantly change the results.
%

\begin{figure}[htbp]
	\begin{center}
		\includegraphics[width=.7\textwidth]{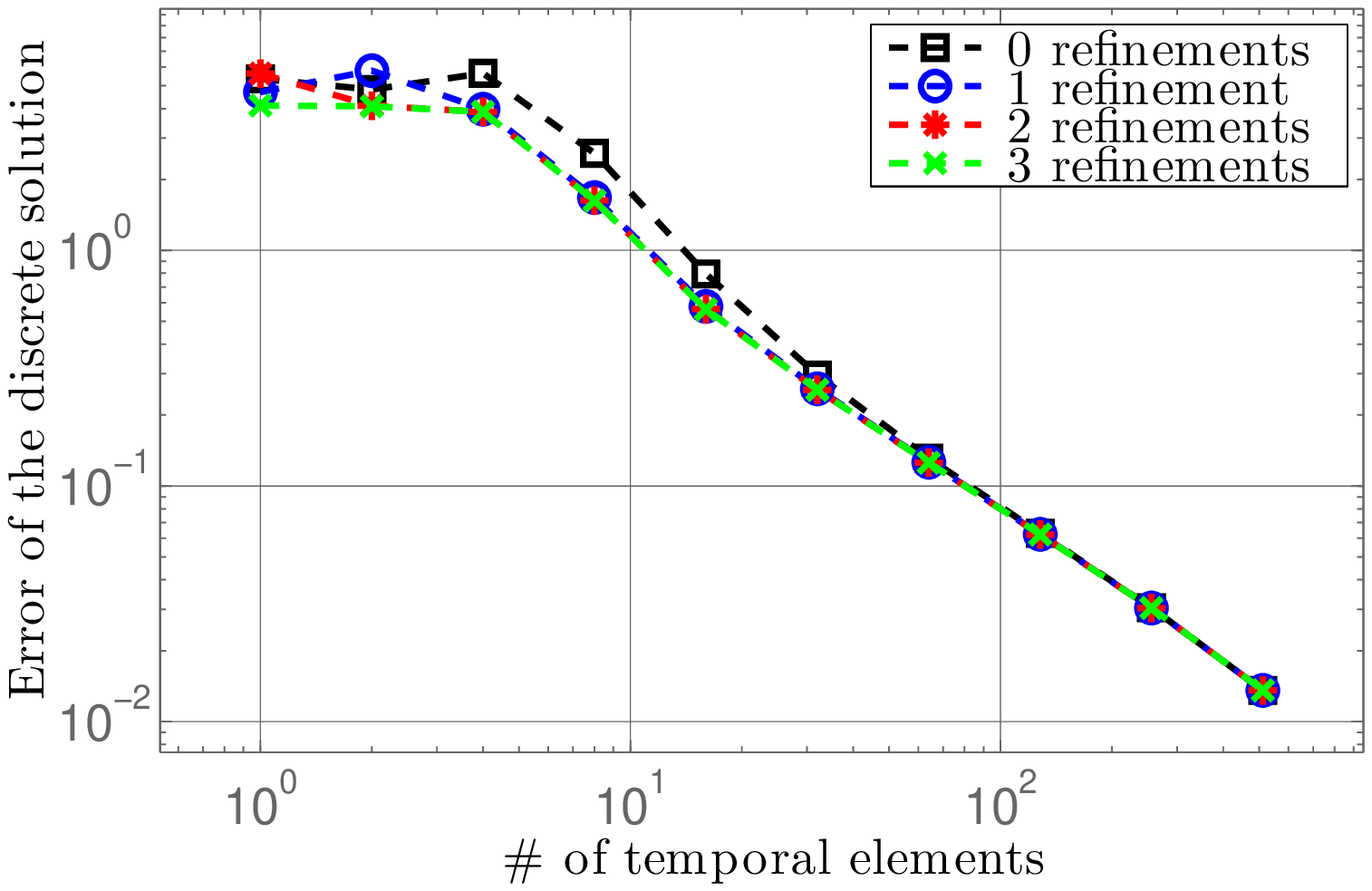}
		\includegraphics[width=.7\textwidth]{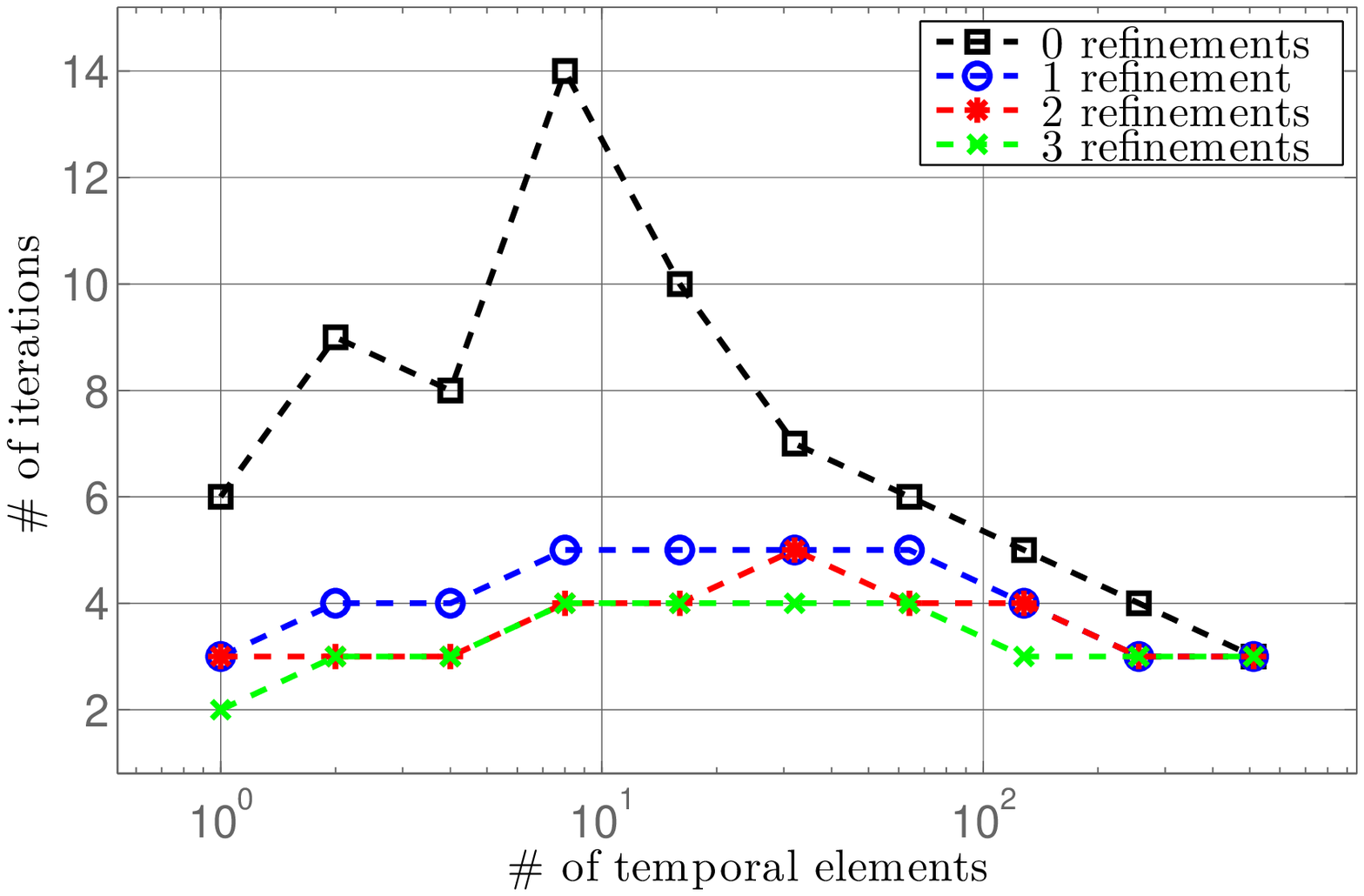}
		\includegraphics[width=.7\textwidth]{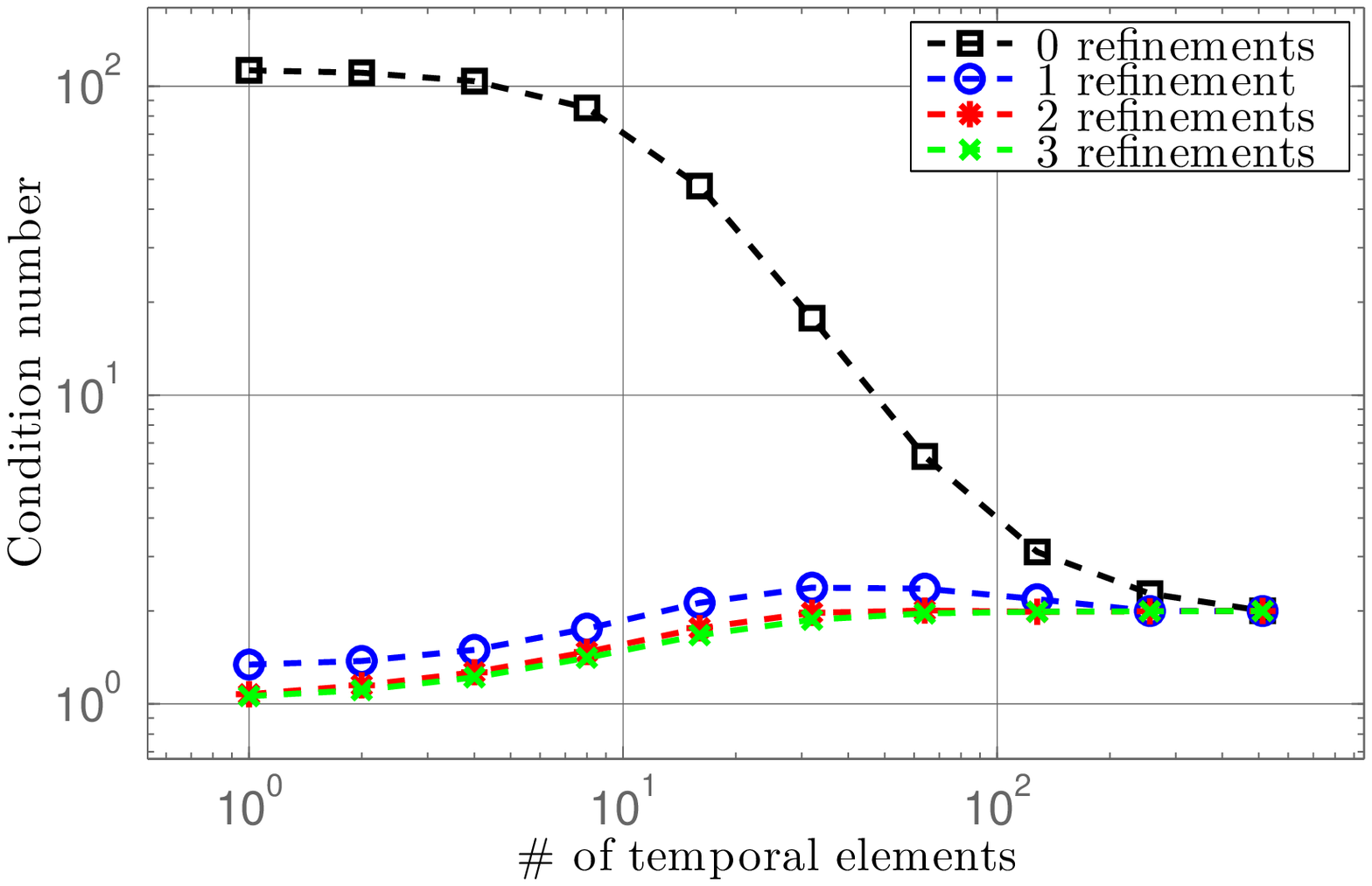}
		\caption{
			a)
			Accuracy of the discrete solution
			in the space-time norm $\tnorm{\cdot}{X}$,
			b)
			the number of 
			generalized LSQR iterations,
			c)
			and
			condition number of the preconditioned system matrix
			as a function of the number
			of temporal elements
			for the setup of \Sec \ref{s:x:B}.
			Each line corresponds to
			a number of temporal refinements
			of the test space.
		}
		\label{f:B}
	\end{center}
\end{figure}

\subsection{Execution times} \label{s:x:C}

In the second experiment
we use the spatial discretization
by 
the first order finite element space
on the L-shaped domain
$D = (-1,1)^2 \setminus [0,1)^2$
produced by the Matlab PDE toolbox
with four subsequent regular mesh refinements.
%
%
This results in $32'705$ spatial degrees of freedom.
For Type 1 and Type 2 temporal subspaces
with different temporal resolutions,
we measure 
a) 
the number of generalized least squares
iterations as in the previous subsection,
b)
the assembly time of the space-time load vector,
and 
c)
the solution time by the generalized LSQR algorithm.
We compare 
to the execution time of
Crank-Nicolson time stepping scheme
with a direct solver for each time step
on the same temporal mesh.
A pool of four Matlab workers
on
a Linux machine equipped
with four AMD Opteron 2220 processors
and 32 GB RAM
was used.
From the results documented in \Fig \ref{f:C} 
we infer that
while
the present implementation
is competitive (in terms of execution time),
it
can only unfold its full potential in 
a massively parallel setting
with 
approximate multigrid- or wavelet-based solvers 
within the space-time preconditioners.

\begin{figure}[htbp]
	\begin{center}
		\includegraphics[width=.7\textwidth]{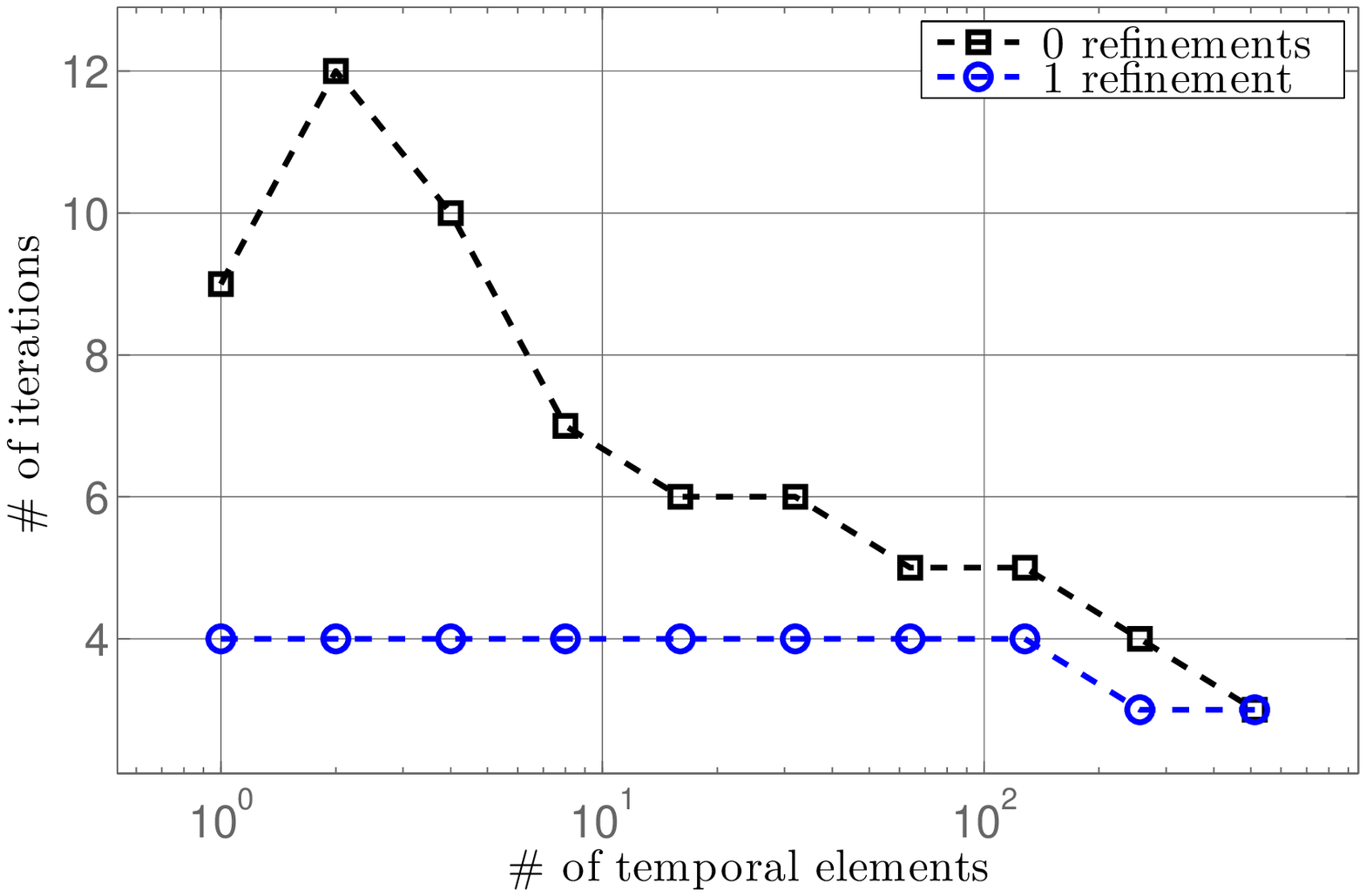}
		\includegraphics[width=.7\textwidth]{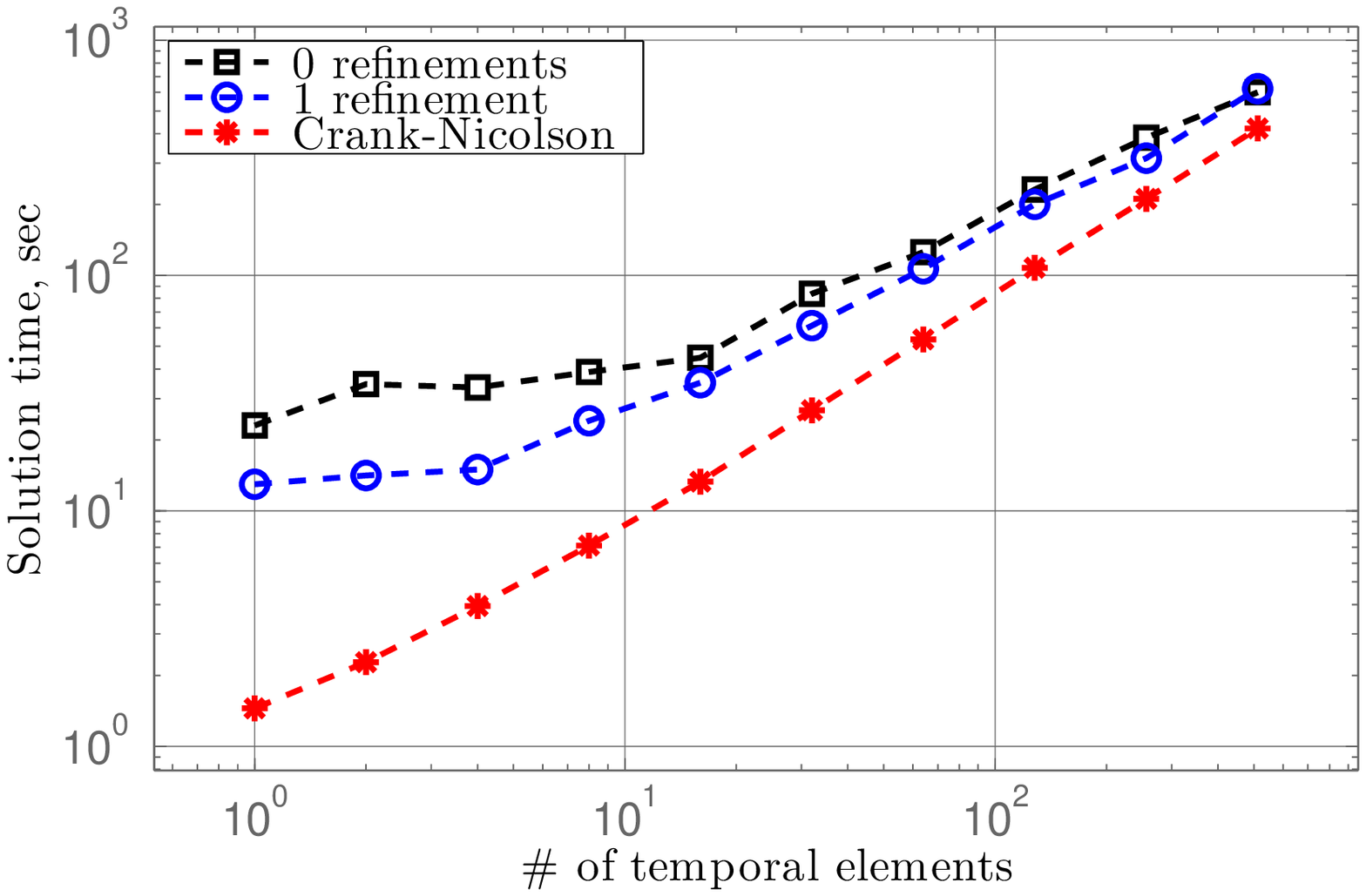}
		\includegraphics[width=.7\textwidth]{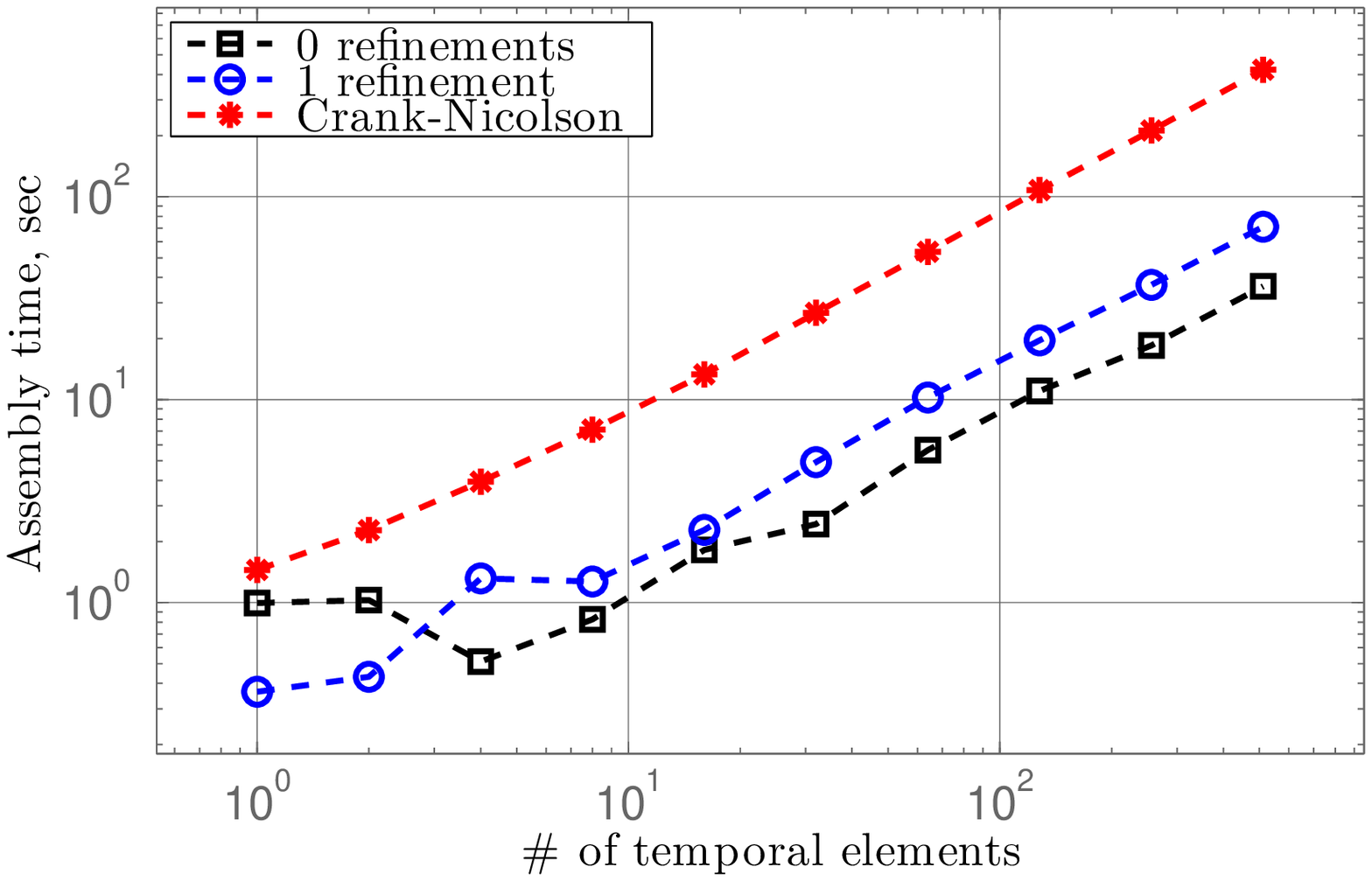}
		\caption{
			a)
			The number of 
			generalized least squares iterations,
			b)
			solution time of the generalized least squares solver,
			c)
			assembly time of the space-time load vector
			as a function of the number
			of temporal elements,
			and in comparison to the Crank-Nicolson time stepping scheme,
			as described in \Sec \ref{s:x:C}.
			Number of refinements
			refers to the number of temporal refinements
			of the test space.
		}
		\label{f:C}
	\end{center}
\end{figure}

\section{Conclusions} \label{s:end}

A concise Matlab implementation
of a 
space-time simultaneous discretization and solution algorithm
for parabolic evolution equations
that is stable, parallelizable, and modular
has been presented.
It admits 
nonuniform temporal meshes and time-dependent inputs.
Very efficient preconditioners
for the iterative resolution
of the resulting single linear system of equations
are available.
Extensions to 
higher order in time
and to 
space-time compressive algorithms
are possible
\cite{Andreev2013,AndreevDiss,AndreevTobler2012}.

Let us point out what we believe
is currently the major obstruction
to massively parallel computations
based on the algorithm presented here.
Allowing arbitrary temporal meshes,
the simultaneous diagonalization
of the temporal FEM matrices 
discussed in \Sec \ref{s:i:inv}
leads to a nonlocal-in-time transformation $\Vbf_t$.
If the solution vector
is split
over multiple computational nodes
along the temporal dimension,
which is natural,
the application of this transformation
may result in heavy communication.
If, however, the temporal mesh is uniform,
or derives from a successive dyadic partition of a coarse mesh,
it may be possible 
to use more efficient
fast Fourier or wavelet based transforms 
for this particular step instead.

\begin{acknowledgements}
	Research in part supported by 
	the Swiss NSF Grant No.~127034
	and by
	the ERC AdG No.~247277
	held by Ch.~Schwab,
	Seminar for Applied Mathematics,
	ETH Z\"urich.
	%
	%
	The authors thanks
	U.S.~Fjordholm
	and
	J.~Schweitzer
	for comments and discussions
	on the draft of this manuscript.
\end{acknowledgements}


\bibliographystyle{spmpsci}
\bibliography{refs}

\begin{thebibliography}{10}
\providecommand{\url}[1]{{#1}}
\providecommand{\urlprefix}{URL }
\expandafter\ifx\csname urlstyle\endcsname\relax
  \providecommand{\doi}[1]{DOI~\discretionary{}{}{}#1}\else
  \providecommand{\doi}{DOI~\discretionary{}{}{}\begingroup
  \urlstyle{rm}\Url}\fi

\bibitem{AkrivisMakridakisNochetto2011}
Akrivis, G., Makridakis, C., Nochetto, R.: {Galerkin and Runge–Kutta methods:
  unified formulation, a posteriori error estimates and nodal
  superconvergence}.
\newblock Numer. Math. \textbf{118}, 429–456 (2011)

\bibitem{AlbertyCarstensenFunken1999}
Alberty, J., Carstensen, C., Funken, S.A.: {Remarks around 50 lines of
  {M}atlab: short finite element implementation}.
\newblock Numer. Algorithms \textbf{20}(2-3), 117–137 (1999)

\bibitem{AndreevDiss}
Andreev, R.: {Stability of space-time Petrov-Galerkin discretizations for
  parabolic evolution equations}.
\newblock Ph.D. thesis, ETH Zürich (2012).
\newblock ETH Diss.~No.~20842

\bibitem{Andreev2013}
Andreev, R.: {Stability of sparse space-time finite element discretizations of
  linear parabolic evolution equations}.
\newblock IMA J. Numer. Anal. \textbf{33}(1), 242–260 (2013)

\bibitem{AndreevTobler2012}
Andreev, R., Tobler, C.: {Multilevel preconditioning and low rank tensor
  iteration for space-time simultaneous discretizations of parabolic {PDEs}}.
\newblock Tech. rep., ETH Zürich (2012).
\newblock In review

\bibitem{BabuskaJanik1989}
Babuška, I., Janik, T.: {The h-p Version of the Finite Element Method for
  Parabolic Equations. I. The p Version in Time.}
\newblock Numer. Methods Partial Differential Equations \textbf{5}, 363–399
  (1989)

\bibitem{BabuskaJanik1990}
Babuška, I., Janik, T.: {The h-p Version of the Finite Element Method for
  Parabolic Equations. II. The h-p Version in Time.}
\newblock Numer. Methods Partial Differential Equations \textbf{6}, 343–369
  (1990)

\bibitem{BanjaiPeterseim2012}
Banjai, L., Peterseim, D.: {Parallel multistep methods for linear evolution
  problems}.
\newblock IMA J. Numer. Anal. \textbf{32}(3), 1217–1240 (2012)

\bibitem{Benbow1999}
Benbow, S.J.: {Solving generalized least-squares problems with {LSQR}}.
\newblock SIAM J. Matrix Anal. Appl. \textbf{21}(1), 166–177 (electronic)
  (1999)

\bibitem{ChangPaigeTitleyPeloquin2009}
Chang, X.W., Paige, C.C., Titley-Peloquin, D.: {Stopping criteria for the
  iterative solution of linear least squares problems}.
\newblock SIAM J. Matrix Anal. Appl. \textbf{31}(2), 831–852 (2009)

\bibitem{Evans1998}
Evans, L.C.: {Partial Differential Equations}, \emph{{Graduate Studies in
  Mathematics}}, vol.~19.
\newblock American Mathematical Society (1998)

\bibitem{Fattorini2005}
Fattorini, H.O.: {Infinite dimensional linear control systems},
  \emph{{North-Holland Mathematics Studies}}, vol. 201.
\newblock Elsevier Science B.V., Amsterdam (2005)

\bibitem{Hulme1972-118}
Hulme, B.L.: {One-step piecewise polynomial {G}alerkin methods for initial
  value problems}.
\newblock Math. Comp. \textbf{26}, 415–426 (1972)

\bibitem{LionsMagenes1972}
Lions, J.L., Magenes, E.: {Non-homogeneous boundary value problems and
  applications. {V}ol. {I}}.
\newblock Springer-Verlag, New York (1972)

\bibitem{PaigeSaunders1982}
Paige, C.C., Saunders, M.A.: {L{SQR}: an algorithm for sparse linear equations
  and sparse least squares}.
\newblock ACM Trans. Math. Software \textbf{8}(1), 43–71 (1982)

\bibitem{SchwabStevenson2009}
Schwab, C., Stevenson, R.: {Space-time adaptive wavelet methods for parabolic
  evolution problems}.
\newblock Math. Comp. \textbf{78}(267), 1293–1318 (2009)

\bibitem{SheenSloanThomee2003}
Sheen, D., Sloan, I.H., Thomée, V.: {A parallel method for time discretization
  of parabolic equations based on {L}aplace transformation and quadrature}.
\newblock IMA J. Numer. Anal. \textbf{23}(2), 269–299 (2003)

\bibitem{ZikatanovXu2003}
Xu, J., Zikatanov, L.: {Some observations on {B}abu\v ska and {B}rezzi
  theories}.
\newblock Numer. Math. \textbf{94}(1), 195–202 (2003)

\end{thebibliography}

\end{document}